\documentclass[a4paper,final]{amsart}
\usepackage{epsfig, epsf, graphicx, float, color}
\usepackage{pstricks, psfrag}
\usepackage{amssymb, amsmath, amsthm, amsfonts, exscale}
\usepackage{amsaddr}
\usepackage{verbatim,enumerate,hyperref}
\usepackage{setspace,mathtools}
\usepackage[numbers,sort&compress]{natbib}
\usepackage{algorithm2e}
\usepackage{steinmetz}
\usepackage{tikz-qtree,tikz-qtree-compat}
\usepackage{tikz,pgf}
\usepackage[margin=1in]{geometry}
\usetikzlibrary{positioning,patterns}
\usetikzlibrary{calc,fadings,decorations.pathreplacing,arrows,
datavisualization.formats.functions,shapes.geometric}
\usepackage[labelfont=bf]{caption}
\allowdisplaybreaks

\def\boxit#1{\vbox{\hrule\hbox{\vrule\kern6pt
          \vbox{\kern6pt#1\kern6pt}\kern6pt\vrule}\hrule}}

\def\var{\hbox{var}}

\def\bse{\begin{eqnarray*}}
\def\ese{\end{eqnarray*}}
\def\be{\begin{eqnarray}}
\def\ee{\end{eqnarray}}
\def\bq{\begin{equation}}
\def\eq{\end{equation}}
\def\bse{\begin{eqnarray*}}
\def\ese{\end{eqnarray*}}

\newcommand{\bbN}{\mathbb{N}}
\newcommand{\bbE}{\mathbb{E}}
\newcommand{\bbF}{\mathbb{F}}

\newcommand{\bbC}{\mathbb{C}}
\newcommand{\bbJ}{\mathbb{J}}

\newcommand{\bR}{\mathbf{R}}

\newcommand{\bI}{\mathbf{I}}

\newcommand{\bG}{\mathbf{G}}
\newcommand{\bW}{\mathbf{W}}
\newcommand{\bP}{\mathbf{P}}

\newcommand{\bE}{\mathbf{E}}

\newcommand{\bJ}{\mathbf{J}}

\newcommand{\bQ}{\mathbf{Q}}
\newcommand{\bS}{\mathbf{S}}
\newcommand{\bV}{\mathbf{V}}

\newcommand{\bx}{\mathbf{x}}
\newcommand{\by}{\mathbf{y}}
\newcommand{\bv}{\mathbf{v}}
\newcommand{\bz}{\mathbf{z}}

\newcommand{\bff}{\mathbf{f}}
\newcommand{\bqq}{\mathbf{q}}

\newcommand{\bw}{\mathbf{w}}

\newcommand{\bg}{\mathbf{g}}

\newcommand{\bbf}{\mathbf{f}}

\newcommand{\btheta}{\boldsymbol{\theta}}

\newcommand{\balpha}{\boldsymbol{\alpha}}

\newcommand{\0}{\mathbf{0}}

\newcommand{\mcA}{{\mathcal A}}
\newcommand{\mcB}{{\mathcal B}}

\newcommand{\mcE}{{\mathcal E}}
\newcommand{\mcF}{\mathcal{F}}
\newcommand{\mcI}{{\mathcal I}}

\newcommand{\mcP}{{\mathcal P}}

\newcommand{\mcS}{\mathcal{S}}

\newcommand{\ii}{\mathbf{i}}

\newcommand{\pp}{\mathbf{p}}

\newcommand{\mm}{\mathbf{m}}

\newcommand{\oone}{\boldsymbol{1}}

\newcommand{\Nset}{\mathbb{N}_0}

\newcommand{\bbNset}{{\mathbb N}}

\newcommand{\Pol}{\mathbb{P}}

\newcommand{\lv}{w}

\newcommand{\Real}{\mathop{\text{\rm Re}}}
\newcommand{\Imag}{\mathop{\text{\rm Im}}}

\newcommand{\func}{u}

\DeclareMathOperator*{\esssup}{ess\,sup}

\newcommand{\sJ}[1]{
\begin{bmatrix*}[r]
  \bJ^{#1}_R   & -\bJ^{#1}_I  \\
  \bJ^{#1}_{I}  & \bJ^{#1}_R   \\
\end{bmatrix*}
}

\newcommand{\szv}[1]{
\begin{bmatrix}
  \bz^{#1}_{R} \\
  \bz^{#1}_{I} \\
\end{bmatrix}
}

\newcommand{\sfv}[1]{
\begin{bmatrix}
  \bbf^{#1}_{R} \\
  \bbf^{#1}_{I} \\
\end{bmatrix}
}

\def\R{\Bbb{R}}

\newtheorem{theo}{Theorem}
\newtheorem{prop}{Proposition}

\newtheorem{asum}{Assumption}

\theoremstyle{remark}
\newtheorem{rem}{Remark}

\begin{document}

\title[Analytic regularity stochastic collocation of Newton iterates]
      {Analytic regularity and stochastic collocation of high
  dimensional Newton iterates.}

\author{Julio E. Castrill\'on-Cand\'as} 
\email{jcandas@bu.edu}

\thanks{This material is based upon work supported by the National
  Science Foundation under Grant No. 1736392.
Research reported in this technical report was supported in part by the
National Institute of General Medical Sciences (NIGMS) of the
National Institutes of Health under award number 1R01GM131409-01.}

\author{Mark Kon}
\address{Department of Mathematics and Statistics \\ Boston University
  \\ 111 Cummington Mall, Boston MA 02215} \email{mkon@bu.edu}


\maketitle

\begin{abstract}
In this paper we introduce concepts from uncertainty quantification
(UQ) and numerical analysis for the efficient evaluation of stochastic 
high dimensional Newton iterates. In particular, we develop complex
analytic regularity theory of the solution with respect to the random
variables. This justifies the application of sparse grids for the
computation of stochastic moments. Convergence rates are derived
and are shown to be subexponential or algebraic with respect to the
number of realizations of random perturbations. Due the accuracy of
the method, sparse grids are well suited for computing low probability
events with high confidence.  We apply our method to the power flow
problem.  Numerical experiments on the 39 bus New England power system
model with large stochastic loads are consistent with the theoretical
convergence rates.
\end{abstract}




\section{Introduction}

Newton iteration is a powerful method for solving many scientific
and engineering problems, naturally arising in the
context of power flow problems \cite{Bergen2000}, non-linear Partial
Differential Equations \cite{Holst1994}, among others.


Computational predictions often form bases for critical decisions in
many areas, including engineering, financial markets, weather
forecasting and disaster management. The rapid development of computer
hardware has allowed simulation of more and more complex
phenomena. But how reliable are these predictions? Can they be
trusted? Rigorous uncertainty quantification (UQ) methods have become
a basic tool for assessing the validity of such computational
predictions.


Uncertainty quantification is a process in which uncertainties in a
system are characterized and propagated into the calculation of a
given Quantity of Interest (QoI). The characterization of
uncertainties arises as a problem often in the solution of so-called
inverse problems, while uncertainty propagation is often associated
with the solution of forward problems. An additional goal of UQ is to
allow for reduction of uncertainties in a QoI by improving the
collection of new data, or by directly improving the physical
understanding itself.  Performing UQ for complex applications such as
electric power grids is a daunting task. This is mainly due to the
fact that it can be very computationally challenging. Effective UQ
processes should be systematic, computationally efficient, and
reliable.

One of the most widely used UQ techniques is the Monte Carlo method
\cite{Fishman1996}, which is robust and easy to implement. Indeed, a
deep analysis or understanding of the underlying stochastic model is
not required, making this an attractive approach for the practicing
engineer and scientist. However, convergence rates for iterative
approximation methods can be very slow. For the application of UQ to
the power flow problem, due to the large numbers of generators, loads
and transmission lines one potentially faces, the problem can be high
dimensional, non-linear, non-Gaussian , and not feasible with current
computational resources. An alternative approach is the use of tensor
product methods. However, these methods suffer significantly from the
{\it curse of dimensionality}, thus making them unattractive even for
moderate dimensionalities.


If the regularity of a QoI is relatively high with respect to the
fundamental random variables, then application of stochastic
collocation with Smolyak sparse grids
\cite{Smolyak63,nobile2008a,nobile2008b} is a good choice. Indeed,
this method has become popular in the field of computational applied
mathematics and engineering as a surrogate model of stochastic Partial
Differential Equations (sPDEs) \cite{Castrillon2016} where the QoI is
composed of moderately large numbers of random variables. The method
is easy to implement and non-intrusive, i.e., each collocation point
corresponds to uncoupled deterministic problems.  The stochastic
collocation method can be used with non-linear dependence of the QoI
on the random variables. However, such grids still suffer from the
curse of dimensionality.  Alternative adaptive techniques have been
developed, including anisotropic sparse grids \cite{nobile2008b},
dimension adaptive quadrature \cite{Gerstner2003} and quasi-optimal
sparse grids \cite{Beck2014,Nobile2016}.  Yet these methods are still
not feasible for very high dimensional problems and/or low regularity
of the QoI. In addition, although quasi-optimal sparse grids lead to
exponential convergence rates with respect to numbers of realizations,
there is to our knowledge still no systematic way to
construct them.


We have particular interest in the application of the Newton iteration
to the solution of the  {\it power flow} equations of  electric grids 
\cite{Bergen2000}. In 
practice many of the generators (wind, solar, etc)  and loads are stochastic in 
nature, and thus a traditional deterministic {\it power flow} analysis is insufficient.
UQ applied to mathematical/statistical modeling of electrical power grids is still in its infancy. A major 2016 National Academy of
Sciences report underscores the importance of this new area of
research in its potential to contribute to this and the next
generation of electric power grids \cite{NAP21919}. 
The incorporation  of uncertainty in the grid has gathered interest in
the power system community. 

In \cite{Hockenberry2004} Hockenberry et
al. proposed the probabilistic collocation method (PCM) with
applications to power systems. Although the results of this work are
good, the uncertainty is computed only with respect to a single
parameter. In \cite{Prempraneerach2010} the authors test a polynomial
chaos collocation and Galerkin approach to study the uncertainty of
power flow on a small 2-bus power system. The results are good for low
stochastic dimensions, but it is not clear how this approach will
scale for large electric power grids, with higher associated
dimensions. Moreover, there is little mathematical theory on the
effectiveness of this approach.

More recently, in \cite{Tang2015}, Tang et al. proposed a
dimension-adaptive sparse grid method \cite{Gerstner2003} by using the
off-the-shelf Matlab Sparse Grid Toolbox \cite{spalg,spdoc}. The
numerical results show the feasibility of this approach.  However,
there are also some weaknesses. The authors did not analyze the
regularity of the power flow with respect to the uncertainty, so that
the rate of convergence of the sparse grid is not known.  Reduced
regularity of the stochastic power flow can choke the accuracy. In
contrast to Monte Carlo methods, which are robust, sparse grid methods
are sensitive to the regularity of the function at hand. Lack of
regularity will lead to erroneous results.

This motivates the application of numerical analysis and UQ theory to
the Newton iteration for such problems as the electric power
grid. Many of the ideas of this paper originate from the numerical
solution of stochastic PDEs
\cite{Castrillon2016,nobile2008a,nobile2008b}, where UQ is having a
large impact. The goals of our paper is to integrate these methods
with the theory and practice of power systems, where we believe that UQ
will eventually have a strong impact. Furthermore, from the numerical
analysis perspective regularity can be determined by complex analytic
extensions of the functions of interest
\cite{Castrillon2016,babusk_nobile_temp_10,Trefethen2012}. These lead
to sharper convergence rates than regularity in terms of derivatives.

In Section \ref{background} the mathematical background for this paper
is introduced. In particular, the Newton-Kantorovich Theorem, sparse
grids and convergence rates are discussed. Furthermore, complex
analysis from the numerical analysis perspective for polynomial
approximation is also treated.  In Section \ref{Analyticity} the
complex analytic regularity of the solution of the Newton iteration is
developed with respect to the random perturbations. In Section
\ref{Applications} the theory developed in Section \ref{Analyticity}
justifies the application of sparse grids to the power flow equations. 
The sparse grids are applied 
to the power flow equations of the 39 bus, 10 Generator, New England model. 
Subexponential or algebraic
convergence rates are obtained than coincide with the sparse grid
convergence rates. These convergence rates make the sparse grid method
suitable for computing stochastic moments to high
accuracy. Furthermore, they will also be well suited for computing
small event tail probabilities with high accuracy.

\section{Mathematical background}
\label{background}

In this section we introduce the general notation and mathematical
background that will be used in this paper. We provide a summary of
the three important topics: i) Newton-Kantorovich Theorem, ii)
Stochastic spaces and iii) Sparse grids (approximation theory).

\subsection{Newton-Kantorovich Theorem}


Consider the Fr\'{e}chet differentiable operator $\bbf:X \rightarrow
Y$ that maps a convex open set $D$ of a Banach space $X$ into a Banach
space $Y$.  Suppose that we are interested in finding an element $\bx
\in D$ such that
\begin{equation}
\bbf(\bx) = \0.
\label{back:eqn1}
\end{equation}
Assuming that a solution of equation \eqref{back:eqn1} exists, a
series of successive approximations $\bx^{v} \in D$, where $v \in
\bbN_{0} := \bbN \cup {0}$, can be built. Consider the space of
bounded linear operators $L(Y,X)$ from $Y$ into $X$.  Let
$\bJ(\bx^{v})$ be the Fr\'{e}chet derivative of
$\bbf(\bx^{v})$. Suppose that $\bJ(\bx^{v})^{-1} \in L(Y,X)$ and
consider the sequence
\[
\bx^{v+1} = \bx^{v} - \bJ(\bx^v)^{-1}\bbf(\bx^v).
\]

\begin{asum}
  We assume that that $\bJ$ satisfies the Lipschitz condition
\begin{equation}
  \| \bJ(\bx) - \bJ(\by) \| \leq \lambda \| \bx - \by
  \|
  \label{newton:eqn2}
  \end{equation}
for some constant $\lambda \geq 0$, and all $\bx,\by \in D$. Furthermore assume that $\bx^0 \in D$ and
there exist positive constants $\varkappa$ and $\delta$ such that
\begin{equation}
  \| \bJ(\bx^0)^{-1} \| \leq \varkappa,
  \| \bJ(\bx^0)^{-1} \bbf(\bx^0)\| \leq \delta,
  h := 2 \varkappa \lambda \delta \leq 1,
  \label{newton:eqn3}
  \end{equation}
  $U(\bx^{0},t^{*}) \subset D$, with $U(\bx,r)$ the open
ball $\{\by:\|\by - \bx\| \leq r \}$ and
$t^{*} = \frac{2}{h}( 1 -
  \sqrt{1 - h})\delta$.
  \label{regions:condition1}
\end{asum}
If Assumption \ref{regions:condition1} above is satisfied, then by the Newton-Kantorovich theorem \cite{Argyros2008},
for all $v \in \bbN_{0}$,
\begin{enumerate}
\item The Newton iterates $\bx^{v+1} = \bx^{v} -
  \bJ(\bx^v)^{-1} \bbf(\bx^{v})$ and $\bJ(\bx^v)^{-1}$ exist.
\item $\bx^v \in U(\bx_0,t^{*}) \subset D$.
\item $\bx^* = \lim_{v \rightarrow \infty} \bx^v$ exists, $\bx^* \in
  \overline{U(\bx^0,t^{*})}$, and $\bbf(\bx^{*}) = \0$ uniquely.
  
\end{enumerate}

\subsection{Stochastic spaces}

Let $\Omega$ be the set of outcomes from the complete probability
space $(\Omega, \mcF, \mathbb{P})$, where $\mcF$ is a sigma algebra of
events and $\mathbb{P}$ is a probability measure.  Define
$L^{q}_{\Pol}(\Omega)$, $q \in [1, \infty]$, as the Banach spaces
\[
L^{q}_{\Pol}(\Omega) := \left\{u:\Omega\rightarrow \R \,\,|\,\, \int_{\Omega}
|\func(\omega)|^{q}\,\mbox{d}\Pol(\omega) < \infty \right\}\,\mbox{and}\,
L^{\infty}_{\Pol}(\Omega) := \left\{u:\Omega\rightarrow\R
\,\,|\,\, \Pol\mbox{-} \esssup_{ \omega \in \Omega}
|\func(\omega)| < \infty \right\}.
\]
\noindent   
Let $\bW:=[W_1, \dots, W_{N}]$ be an
$N$-component random vector measurable in $(\Omega, \mcF, \mathbb{P})$
that takes values on $\Gamma:=\Gamma_{1} \times \dots \times
\Gamma_{N} \subset \mathbb{R}^{N}$, with $\Gamma_n:=[-1,1]$. Let
$\mcB(\Gamma)$ be the Borel $\sigma$-algebra.  Define the induced
measure $\mu_{\bW}$ on $(\Gamma,\mcB(\Gamma))$ as $\mu_{\bW}(A) : =
\mathbb{P}(\bW^{-1}(A))$ for all $A \in \mcB(\Gamma)$. Assuming that
the induced measure is absolutely continuous with respect to 
Lebesgue measure on $\Gamma$, there exists a density
function $\rho(\bqq): \Gamma \rightarrow [0, +\infty)$ such that for
  any event $A \in \mcB(\Gamma)$
\[
\mathbb{P}(\bW \in A) := \mathbb{P}(\bW^{-1}(A)) = \int_{A} \rho( \bqq
)\,\mbox{d} \bqq.
\]
\noindent Now, for any measurable function $\bW \in
          [L^{1}_{\mathbb{P}} (\Gamma)]^{N}$ define the expected value
          as
\[
\mathbb{E}[\bW] = \int_{\Gamma} \bqq \, \rho( \bqq )\, \mbox{d}\bqq.
\]
Define also the Banach spaces (for $q \geq 1$)
\[
\begin{split}
L^{q}_{\rho}(\Gamma) 
&:= \left\{ u:\Gamma \rightarrow \R \,\,|\,\, \int_{\Omega}
|\func(\bqq)|^{q}\, \rho(\bqq) \,\mbox{d}\bqq
< \infty \right\} \,\,\mbox{and}\,\,\,
L^{\infty}_{\rho}(\Gamma) 
:= \left\{u:\Gamma \rightarrow \R  \,\,|\,\, \rho\mbox{-}\esssup_{ \bqq \in \Gamma}
|\func(\bqq)| < \infty \right\}.
\end{split}
\]
Note that the above $\rho$-essential supremium is
with respect to the measure induced by the density function $\rho(\cdot)$, rather than Lebesgue measure itself.

In general the density $\rho(\cdot)$ will not factorize into independent
probability density functions, making higher dimensional manipulations difficult in some cases.  In \cite{babusk_nobile_temp_10} the authors
recommend use of an auxiliary probability density function
$\hat\rho:\Gamma\rightarrow \R^+$ that factorizes into $N$ independent
ones, i.e.,
\begin{equation}\label{eq:rhohat}
  \hat\rho(\bqq) = \prod_{n=1}^{N} \hat\rho(q_n), \;\;\; \bqq=(q_1,\ldots,q_N)\in\Gamma,
\end{equation}
where we will assume that $\left\|\frac{\rho(\bqq)}{\hat
  \rho(\bqq)}\right\|_{L^\infty(\Gamma) }<\infty$. Note that in
contrast $L^\infty(\Gamma)$ (without a subscript) is with respect to the supremium in the
Lebesgue measure.

\subsection{Sparse grids}
\label{sparsegrids}

Our objective is to efficiently approximate a function $\func: \Gamma
\rightarrow \R$ defined on high dimensional domains using global
polynomials. The accuracy of the approach will directly depend on the
regularity of the function.  Let ${ \mcP_{\pp}(\Gamma)} \subset
L^2(\Gamma)$ be the span of tensor product polynomials of degree at
most $ {\pp = (p_1,\ldots,p_N)}$; i.e., $ \mcP_{\pp}(\Gamma) =
\bigotimes_{n=1}^{N}\;\mcP_{ p_n}(\Gamma_{n})$ with $\mcP_{
  p_n}(\Gamma_{n}):=\text{\rm span} (q_n^m,\,m=0,\dots,p_n),$ $\quad
n=1,\dots,N$.
For univariate polynomial approximation, we define a sequence of levels $i=0,1,2,3,\ldots$
corresponding to increasing degrees $m(i)\in \bbN_0$ polynomial approximation for given coordinates. Here for a 
given approximation scheme, $m(\cdot)$ is a fixed function. In general our multivariate 
approximation scheme will assume different levels of approximation $i_n$ for different 
coordinates $n=1,\ldots,N$.

We consider separate univariate Lagrange interpolants 
in $\Gamma$ along each dimension $n$, given as ${\mathcal I}^{m(i_n)}_{n}:C^{0}(\Gamma_n)
\rightarrow {\mathcal P}_{m(i_n)-1}(\Gamma_n)$. Specifically, let
\begin{equation}
{\mathcal I}^{m(i_n)}_{n}(\func(q_n)):=
\sum_{j_n=1}^{m(i_n)}\func(q_{j_n}^{i})l_{n,j_n}(q_n),
\label{sparsegrid:eqnm1}
\end{equation}
where $\{ l_{n,j} \}_{j=1}^{m(i_n)}$ is a Lagrange basis for the space
${\mathcal P}_{p_n}(\Gamma_{n})$, the set $\{q^{i}_{j_n}\}_{j_n=1}^{m(i_n)}$, represents $m(i_n)$ discrete locations (the interpolation
knots) in $\Gamma_n$, the index $i_n \geq 0$ is the level of approximation, and
$m(i_n) \in \bbN_{+}$ is the number of collocation nodes at level $i_n
\in \bbN_{+}$ where $m(0) = 0$, $m(1) = 1$ and $m(i_n) \leq m(i_n+1)$
if $i_n \geq 1$.

\begin{rem} Since $m(i_n)$ represents the number of interpolation
  knots for the Lagrange basis, we have $p_n = m(i_n) + 1$.
  \end{rem}

One of the most common approaches to constructing Lagrange
interpolants in high dimensions is the formation of tensor products of ${\mathcal
  I}^{m(i_n)}_n$ along each dimension $n$. However, as $N$ increases
the dimension of $\mcP_p$ increases as $\prod_{n=1}^N (p_n+1)$.  Thus
even for moderate dimensions $N$ the computational cost of a Lagrange
approximation becomes intractable.  However, in the case of sufficient
complex analytic regularity of the function $\func$ with respect to
the random variables defined on $\Gamma$, the application of Smolyak
sparse grids is better suited. In the rest of this section the
construction of the classical Smolyak sparse grid (see e.g.
\cite{Smolyak63,Novak_Ritter_00}) is summarized. More details can be
found in \cite{Back2011}.

Consider the difference operator along the $n^{th}$ dimension given by
\begin{equation}
  {\Delta_n^{m(i_n)} :=} \mcI_n^{m(i_n)}-\mcI_n^{m(i_n-1)}.
\label{sparsegrid:eqn0}
\end{equation}
Given an integer $w \geq 0$, called the approximation level, and a
multi-index $\ii=(i_1,\ldots,i_{N})$ $\in \bbNset^{N}_0$, let
$g:\bbNset^{N}_0 \rightarrow \bbNset$ be a strictly increasing
function in each argument and define a sparse grid approximation of
function $\func(\bqq) \in C^{0}(\Gamma)$, restricted in order by $g$:

\begin{equation}
  \mcS^{m,g}_{w}[\func(\bqq)]
=
 \sum_{\ii\in\bbNset^{N}_0: g(\ii)\leq w} \;\;
 \bigotimes_{n=1}^{N} ({\Delta_n^{m(i_n)}})(\func(\bqq)). 
\label{sparsegrid:eqn1}
\end{equation}



From the previous expression, the sparse grid approximation is
obtained as a linear combination of full tensor product
interpolations.  However, the constraint $g(\ii)\leq w$ in
\eqref{sparsegrid:eqn1} restricts the use of tensor grids of high
degree.

Let $\mm(\ii) = (m(i_1),\ldots,m(i_{N}))$ and consider the set of
polynomial multi-degrees
\[
\Lambda^{m,g}(w) = \{\pp\in\bbNset^{N}, \;\;  g(\mm^{-1}(\pp+\oone))\leq w\},
\]
where $\oone$ is an $N$ dimensional vector of ones. Denote by
$\Pol_{\Lambda^{m,g}(w)}(\Gamma)$ the corresponding multivariate
polynomial space spanned by the monomials with multi-degree in
$\Lambda^{m,g}(w)$, i.e.
\[
\Pol_{\Lambda^{m,g}(w)}(\Gamma) = \text{\rm span}
\left\{\prod_{n=1}^{N}
q_n^{p_n}, \;\; \text{with } \pp\in\Lambda^{m,g}(w)\right\}.
\]

For a Banach space $V$ let
\[
C^{0}(\Gamma; V) : = \{ \func: \Gamma \rightarrow V\,\, \mbox{is
  continuous on $\Gamma$ and } \max_{y\in \Gamma} \|\func(y)\|_{V} <
\infty \}.
\]
It can be shown that the approximation formula given by
$\mcS^{m,g}_{w}$ is exact in $\Pol_{\Lambda^{m,g}(w)}(\Gamma)$. We
state the following proposition that is proved in \cite{Back2011}.
\begin{prop}
\label{prop1}
\
\begin{itemize} 
\item[a)] For any $\func \in C^0(\Gamma;V)$, we have
  $\mcS_{w}^{m,g}[\func]\in \Pol_{\Lambda^{m,g}(w)}\otimes V$.
\item[b)] Moreover, $\mcS_{w}^{m,g}[\func] = \func  \;\; \forall
  \func \in\Pol_{\Lambda^{m,g}(w)}\otimes V$.
\end{itemize}
\end{prop}

\begin{rem} The tensor product space $\Pol_{\Lambda^{m,g}(w)}\otimes V$
  is more easily understood as the space of polynomials with
  Banach-valued coefficients. Furthermore, $\mcS_{w}^{m,g}[\func]\in
  \Pol_{\Lambda^{m,g}(w)}\otimes V$ is interpreted as a sparse grid
  approximation of a $V$-valued continuous function.
  \end{rem}

A good choice of $m$ and $g$ is given by the Smolyak sparse grid definitions (see
\cite{Smolyak63,Novak_Ritter_00})
\[
m(i_n) = \begin{cases} 1 & \text{for } i_n=1 \\ 2^{i_n-1}+1 &
  \text{for } i_n>1 \end{cases}\quad \text{ and } \quad g(\ii) =
\sum_{n=1}^N (i_n-1).
\]
Furthermore $ \Lambda^{m,g}(w):=\{\pp\in\bbNset^{N}: \;\; \sum_n
f(p_n) \leq w \}$ where
\[
    f(p_n) = \begin{cases}
      0, \; p_n=0 \\
      1, \; p_n=1 \\
      \lceil \log_2(p_n) \rceil, \; p_n \geq 2
    \end{cases}.
\]
Other common choices are shown in Table \ref{multivariate:table2}.
\begin{table}[htp]
\begin{center}
 \begin{tabular}{c c c}
    \hline
    Approx. space  & sparse grid: \;\; $m$, $g$	     & polynomial space: \;\; $\Lambda(w)$ \\
    \hline 
    Total          & $m(i_n)=i_n$ 			     &$\{\pp\in\Nset^{N}: \;\; \sum_n p_n \leq w\}$ \\
    Degree (TD)    & $g(\ii) = \sum_n(i_n-1) \leq w$ &	\\
    & & \\
    Hyperbolic 	   &	$m(i)=i$ 		     &	$\{\pp\in\Nset^{N}:$ \\
    Cross (HC)     &	$g(\ii) = \prod_n(i_n) \leq w+1$ & $\prod_n(p_n+1) \leq w+1\}$ \\
    \hline
    \end{tabular} 
\end{center}
\caption{Sparse grid approximations formulas for TD and HC.}
\label{multivariate:table2}
\end{table}

This choice of $m$ and $g$ combined with the choice of Clenshaw-Curtis
abscissas (which are locations of interpolation points given as
extrema of Chebyshev polynomials) leads to nested sequences of one
dimensional interpolation formulas and a sparse grid with a highly
reduced number of nodes compared to the corresponding tensor
grid. Another good choice includes Gaussian abscissas
\cite{nobile_tempone_08}. For any choice of $m(i_n) > 1$ the
Clenshaw-Curtis abscissas are given by
\[
q^{i_n}_{j_n} = -\cos \left( \frac{\pi(j_n-1)}{m(i_n) - 1} \right),\,\, j_n =
1,\dots, m(i_n).
\]
In Figure \ref{sparsegrid:fig1} an example of Clenshaw Curtis and
Gaussian abscissas are shown for $w = 5$.
\begin{figure}[hb]
  \begin{center}
    \begin{tikzpicture}[thick,scale=0.9, every node/.style={scale=0.9}]

      \node[inner sep=0pt] (russell) at (0,0)
           {\includegraphics[width=0.99\textwidth,clip,
               trim=1cm 3.25cm 1cm 2.85cm]
    {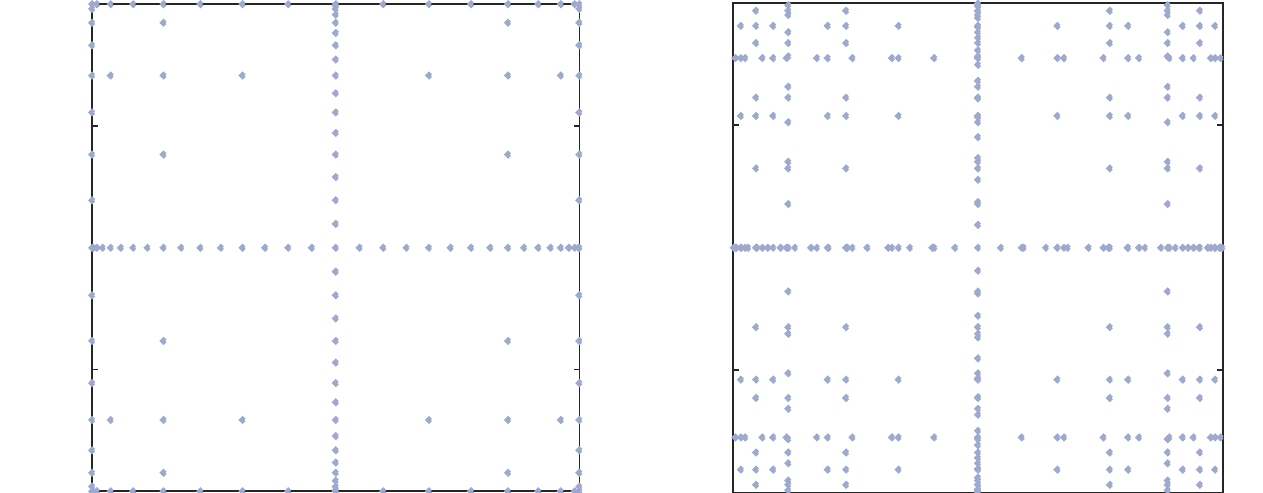}};

      \node [] at (-4.25,-4.5) {(a) Clenshaw Curtis};
      \node [] at (4.75,-4.5) {(b) Gaussian};
    \normalsize
    \end{tikzpicture}
  \end{center}

  \caption{Clenshaw Curtis (left) and Gaussian abscissas (right) for
    $w = 5$ levels.}
  \label{sparsegrid:fig1}
\end{figure}

As previously pointed out, the probability density function $\rho$ does
not necessarily factorize in higher dimensions. As an alternative we
use the auxiliary distribution $\hat \rho$, which factorizes as
$\hat\rho(\bqq) = \prod_{n=1}^{N} \hat\rho_{n}(q_n)$ and is close to
the original distribution $\rho(\bqq)$.  Suppose $k$ is a given global index determined by 
the set of indices $k_1 \dots, k_N$ as $k = k_1 + p_1(k_2-1) + p_1p_2(k_3-1) +
p_1p_2p_3(k_4 - 1) + \dots$. Given a function $u:\Gamma \rightarrow
V$, the quadrature scheme $\mathbb{E}_{\hat\rho}^{\pp}[u]$ that
approximates the integral $\mathbb{E}[u(\bqq)]:= \int_{\Gamma}
u(\bqq) \hat \rho(\bqq) \,d \bqq$ can now be computed based on the distribution $\hat\rho(\bqq)$ 
as
\[
   \mathbb{E}_{\hat\rho}^{\pp}[u] = \sum_{k=1}^{N_{\pp}} \omega_k
   u(\bqq^{(k)}), \quad \omega_k = \prod_{n=1}^N \omega_{k_n}
   \quad
   \omega_{k_n} = \int_{\Gamma_{n}}
   l_{n,k_n}^2(q_n)\hat\rho_n(q_n)\,\mbox{d}q_n,
   \]
   with $\bqq^{(k)} \in \R^{n}$ the locations of the quadrature
   knots, and $N_\pp \in \mathbb{N}$ the number of Gauss quadrature
   points controlling the accuracy of the quadrature scheme.  Recall that $l_{n,k_n}^2(q_n)$ define Lagrange polynomials defined in equation \eqref{sparsegrid:eqnm1}.  The
   term $\mathbb{E}[\func(\bqq)]$ can be approximated as
\[
\mathbb{E}[\mcS^{m,g}_{\lv}[\func(\bqq)]] \approx
\mathbb{E}^{\pp}_{\hat{\rho}}[\mcS^{m,g}_{\lv} [\func(\bqq)] \frac{\rho}{\hat{\rho}}],
\label{sparsegrid:eqn2} 
\]
and similarly the variance $\var[\func(\bqq)]$ is approximated as
\[
\begin{split}
\var[\func(\bqq)] 
& \approx \mathbb{E}[ (\mcS^{m,g}_{\lv} [\func(\bqq)])^2 ] 
- \mathbb{E}[\mcS^{m,g}_{\lv} [\func(\bqq)]]^{2} \\
& \approx
\mathbb{E}^{\pp}_{\hat{\rho}}
       [ (\mcS^{m,g}_{\lv} [\func(\bqq)])^2 \frac{\rho}{\hat{\rho}}] 
       - \mathbb{E}^{\pp}_{\hat{\rho}}
       [\mcS^{m,g}_{\lv} [\func(\bqq)] \frac{\rho}{\hat{\rho}}]^{2}.
\end{split}
\]
\begin{rem}
The weights $\omega_{k_n}$ and node locations $\bqq^{(k)}$ are
computed from the auxiliary density $\hat \rho$. For standard
distributions of $\hat \rho$ such as uniform and Gaussian, these are
already tabulated to full accuracy. Otherwise they must be
computed by solving for roots of orthogonal polynomials and using a
quadrature scheme. However, the integrals involved are only one dimensional. 
See \cite{babusk_nobile_temp_10} (Section 2) for details.
\end{rem}

We now develop some rigorous numerical bounds for the accuracy of the
sparse grid approximation.  Let $C^{k}_{\rm mix}(\Gamma;
\R)$ denote the space of functions with continuous mixed derivatives up to degree $k$:

\[
C^k_{\rm mix}(\Gamma; \R) = \left\{ u:\Gamma \rightarrow \R :
\frac{\partial^{\alpha_1, \dots, \alpha_N} u}{\partial^{\alpha_1} q_1
  \dots \partial^{\alpha_N} q_N} 
\in C^0(\Gamma; \R),\,n = 1,\dots,N,\,
\alpha_{n} \leq k \right\}
\]
and equipped with the following norm:
\[
\|\func\|_{C^k_{\rm mix}(\Gamma; \R)} = \left\{\func:\Gamma \rightarrow \R: 
\max_{\bqq \in \Gamma}
\left|
\frac{\partial^{\alpha_1, \dots, \alpha_N} \func(\bqq)}{\partial^{\alpha_1} q_1
  \dots \partial^{\alpha_N} q_N}    
\right| < \infty
\right\}.
\]
Assume that $\func\in C^{k}_{\rm mix}(\Gamma;\R)$.
In \cite{Novak_Ritter_00} the authors show that it must follow that 
\[
\|\func - \mcS^{m,g}_{w}[\func]\|_{L^{\infty}(\Gamma)}
\leq C(k,N)  \|\func\|_{C^k_{\rm mix}(\Gamma)}
\eta^{-k} (\log{\eta})^{(k+2)(N-1)+1},
\]
where $\eta$ is the number of knots of the sparse grid $\mcS^{m,g}_{w}$.
However, the coefficient $C(k,N)$ is in general not known
\cite{Novak_Ritter_00}.

If the function $\func$ admits a complex analytic extension, a better
approach for deriving error bounds for the polynomial approximation arises
from exploitation of analysis in the complex plane. In \cite{nobile2008a} the authors derive
$L^{\infty}_{\rho}(\Gamma)$ bounds based on analytic extensions of
$\func$ on a well defined region $\Psi \subset \bbC^{N}$ with respect
to the variables $\bqq$. These bounds are explicit, and the
coefficients can be estimated and depend on the size of the region
$\Psi$.

In \cite{nobile2008b,nobile2008a} the authors derive error estimates
for isotropic and anisotropic Smolyak sparse grids with
Clenshaw-Curtis and Gaussian abscissas, with error $\|\func -
\mcS^{m,g}_{w}[\func]\|_{L^{\infty}_{\rho}(\Gamma)}$ exhibiting
algebraic or sub-exponential convergence with respect to the number of
collocation knots $\eta$ (see Theorems 3.10, 3.11, 3.18 and 3.19 in
\cite{nobile2008a} for more details). However, for these estimates to
be valid, $\func \in C^{0}(\Gamma,\R)$ has to admit an analytic
extension to the region defined by the following polyellipse in
$\bbC^{N}$: ${\mathcal E}_{\hat \sigma_1, \dots, \hat \sigma_{N}} : =
\Pi_{n=1}^{N}{\mathcal E}_{n,\hat \sigma_n} \subset \Psi$, where
\[
\begin{split}
  {\mathcal E}_{n,\hat \sigma_n} &= \left\{ z \in \bbC \mathrm{\ with}\,\Real{z} =
  \frac{e^{\delta_n} + e^{-\delta_n} }{2}\cos(\theta),\,\,\,\Imag{z} =
  \frac{e^{\delta_n} - e^{-\delta_n}}{2}\sin(\theta): \theta \in
       [0,2\pi) , \hat \sigma_n \geq \delta_{n} \geq 0 \right\}
  \end{split}
\]
and $\hat \sigma_n > 0$ (see the Bernstein ellipse in Figure
\ref{erroranalysis:sparsegrid:polyellipse}).
\begin{figure}[htp]
\begin{center}
\begin{tikzpicture}
    \begin{scope}[font=\scriptsize]



      
    \filldraw[fill=blue!40, 
      semitransparent] (0,0) ellipse (2 and 1);

    \draw [->] (-2.5, 0) -- (2.5, 0) node [above left]  {$\Real $};



    \draw [->] (0,-1.5) -- (0,1.5) node [below right] {$\Imag$};
    \draw (1,-3pt) -- (1,3pt)   node [above] {$1$};
    \draw (-1,-3pt) -- (-1,3pt) node [above] {$-1$};
    \end{scope}
    
    \node [below right] at (1.50,1.25) {${\mathcal E}_{n,\hat \sigma_n}$}; 
\end{tikzpicture}
\end{center}
\caption{Bernstein ellipse along the $n^{th}$ dimension. The ellipse crosses
  the real axis at $\frac{e^{\hat \sigma_n} + e^{-\hat \sigma_n}}{2}$
  and the imaginary axis at $\frac{e^{\hat \sigma_n} - e^{-\hat
      \sigma_n}}{2}$.}
\label{erroranalysis:sparsegrid:polyellipse}
\end{figure}
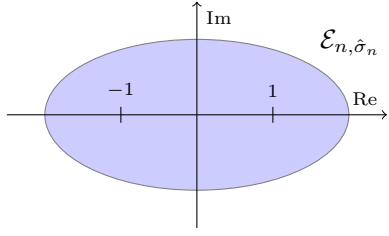


We now recall the definition of Chebyshev polynomials, useful in deriving error estimates on sparse grids.  Let
$T_k:\Gamma_{1} \rightarrow \R$, $k = 0, 1, \dots$, be a $k^{th}$
order Chebyshev polynomial over $[-1,1]$. These polynomials are defined
recursively as:
\[
\begin{split}
T_0(y) = 1,  T_1(y) = y, \dots,
T_{k+1}(y) = 2yT_{k}(y) - T_{k-1}(y), \dots \end{split}\ . 
\]
The following theorem characterizes  approximation of analytic
functions using Chebyshev polynomials.
\begin{theo} Let $\func$ be analytic and absolutely
  bounded by $M$ on ${\mathcal E}_{\log \zeta}$, $\zeta > 1$. Then the
  expansion
\[
\func(y) = \alpha_0 + 2\sum_{k = 1}^{\infty} \alpha_{k} T_{k}(y), 
\]
holds for all $y \in {\mathcal E}_{\log \zeta}$ where
\[
\alpha_k = \frac{1}{\pi} \int_{-1}^{1} \frac{\func(y)T_k(y)}{1 - y^2}\, dy
\]
and, additionally, $|\alpha_k| \leq M / \zeta^k$. Furthermore if $y \in
[-1,1]$ then
\[
|\func(y) - \alpha_0  - 2\sum_{k = 1}^{m} \alpha_{k} T_{k}(y)|
\leq 
\frac{2M}{\zeta - 1} \zeta^{-m}.
\]
\label{errorestimates:theorem}
\end{theo}
\begin{proof}
See Theorem 8.2 in \cite{Trefethen2012}
\end{proof}
  
We follow the arguments in \cite{nobile_tempone_08, babusk_nobile_temp_10} and using the fact that the
interpolation operator ${\mathcal I}^{m(i_n)}_{n}$ is exact on the space
${\mathcal P}_{p_n - 1}$, i.e. for any $v \in {\mathcal P}_{p_n - 1}$ 
we have that ${\mathcal I}^{m(i_n)}_{n}(v) = v$, it can be shown that if $\func$ 
is continuous on
$[-1,1]$ and has an analytic extension on ${\mathcal E}_{\sigma_n}$ we
have, from Theorem \ref{errorestimates:theorem},
\[
\|(I - {\mathcal I}^{m(i_n)}_{n})\func\|_{L^{\infty}_{\rho}(\Gamma_n)} \leq
(1 + \Lambda_{m(i)})
\min_{v \in {\mathcal P}_{m(i_n) - 1}}  \| u - v\|_{C^{0}(\Gamma,\R)}
\leq
(1 + \Lambda_{m(i)})
\frac{2M(\func)}{e^{\sigma_n} - 1} e^{-\sigma_nm(i_n)},
\]
where $\Lambda_{m(i_n)}$ is the Lebesgue constant and is bounded by
$2\pi^{-1}(\log{(m-1)} + 1)$ (see \cite{babusk_nobile_temp_10}) and $M
= M(u)$ is the maximal value of $u$ on ${\mathcal E}_{\sigma_n}$ .  Thus,
for $n = 1,\dots,N$ we have
\begin{equation}
\begin{split}
\|(I - {\mathcal I}^{m(i)}_{n})\func\|_{L^{\infty}_{\rho}(\Gamma_n)
}
&\leq
M(\func)
C(\hat \sigma_n)
i_n e^{-\sigma_n 2^{i_n}}
\end{split},
\label{errorestimates:eqn3}
\end{equation}
where $\sigma_n = \frac{\hat \sigma_n}{2} > 0$ and $C(\sigma_n) :=
\frac{2}{(e^{ \sigma_n} - 1)}$. Recalling the definition of
${\Delta_n^{m(i_n)}}$ from equation \eqref{sparsegrid:eqn0}, we
have that for all $n = 1,\dots,N$
\begin{equation}
\begin{split}
  \| \Delta(\func)^{m(i_n)} \|_{L^{\infty}_{\rho}
    (\Gamma_n)} 
&=
\|
({\mathcal I}^{m(i_n)}_{n} - {\mathcal I}^{m(i_n-1)}_{n})\func
\|_{L^{\infty}_{\rho}(\Gamma_n)}
\leq
\|(I - {\mathcal I}^{m(i_n)}_{n})\func\|_{L^{\infty}_{\rho}(\Gamma_n)} \\
&+
\|(I - {\mathcal I}^{m(i_n-1)}_{n})\func\|_{L^{\infty}_{\rho}(\Gamma_n)}
\leq
2 M(\func)
C(\sigma_n)
i_n e^{-\sigma_n 2^{i_n-1}}.
\end{split}
\label{errorestimates:eqn4}
\end{equation}
By applying equation \eqref{errorestimates:eqn4} to Lemma 3.5 in \cite{nobile2008a},
we are now in a position to slightly modify Theorems 3.10 and 3.11 in \cite{nobile2008a}
and restate them into a single theorem given below. However, the following assumptions and definitions
are first needed:

\begin{itemize}

\item We set $\hat \sigma \equiv \min_{n=1,\dots,N} \hat \sigma_n,$ i.e. for an
  isotropic sparse grid the overall asymptotic sub-exponential decay
  rate $\hat \sigma$ will be dominated by the smallest $\hat
  \sigma_n$.

\item Let
\[
\tilde{M}(\func) = \sup_{\bg \in {\mathcal E}_{\hat \sigma_1, \dots, \hat
    \sigma_{N}}} |\func(\bg)|,
\]
$\sigma = \hat{\sigma}/2$, $\mu_1 = \frac{\sigma}{1 + \log (2N)}$, and
$\mu_2(N) = \frac{\log(2)}{N(1 + \log(2N))}$,
\[
a(\delta,\sigma):=
\exp{
\left(
\delta \sigma \left\{
\frac{1}{\sigma \log^{2}{(2)}}
+ \frac{1}{\log{(2)}\sqrt{2 \sigma}}
+ 2\left( 1 + \frac{1}{\log{(2)}} 
\sqrt{ \frac{\pi}{2\sigma} }
\right)
\right\}
\right)
},
\]
\[
\begin{split}
\tilde{C}_{2}(\sigma) &= 1 + \frac{1}{\log{2}}\sqrt{
\frac{\pi}{2\sigma}}
,\,\,\delta^{*}(\sigma) = \frac{e\log{(2)} - 1}{\tilde{C}_2
  (\sigma)}, \,\,C_1(\sigma,\delta,\tilde M(\func))
=
\frac{4\tilde{M}(\func)
C(\sigma)a(\delta,\sigma)}{
  e\delta\sigma},
\end{split}
\]
$\mu_3 = \frac{\sigma \delta^{*} \tilde C_2(\sigma)}{1 + 2 \log
  (2N)}$, and
\[{\mathcal Q}(\sigma,\delta^{*}(\sigma),N, \tilde M(\func)) = 
\frac{ C_1(\sigma,\delta^{*}(\sigma),\tilde M(\func))}{\exp(\sigma
  \delta^{*}(\sigma) \tilde{C}_2(\sigma) )}
\frac{\max\{1,C_1(\sigma,\delta^{*}(\sigma),\tilde M(\func))\}^{N}}{|1
  - C_1(\sigma,\delta^{*}(\sigma),\tilde M(\func))|}.
\]
\end{itemize}

\begin{theo}
Suppose that $\func \in C^{0}(\Gamma;\R)$ has an analytic extension on
${\mathcal E}_{\hat \sigma_1, \dots, \hat \sigma_{N}}$ and is absolutely
bounded by $\tilde M(\func)$.  If $w > N / \log{2}$ and a sparse grid
with Clenshaw-Curtis abscissas is used, then the following bound is
valid:
\begin{equation}   
\|\func - \mcS^{m,g}_{w}\func
\|_{L^{\infty}_{\rho}(\Gamma)} \leq 
{\mathcal
  Q}(\sigma,\delta^{*}(\sigma),N, \tilde M(\func))
\eta^{\mu_3(\sigma,\delta^{*}(\sigma),N)}\exp
  \left(-\frac{N \sigma}{2^{1/N}} \eta^{\mu_2(N)} \right)
, \\
\label{erroranalysis:sparsegrid:estimate}
\end{equation}
Furthermore, if $w \leq N / \log{2}$ then the following algebraic
convergence bound holds:
\begin{equation}
  \begin{split}
    \| \func - \mcS^{m,g}_{w}\func \|_{L^{\infty}_{\rho}(\Gamma)}
    &\leq 
\frac{C_1(\sigma,\delta^{*}(\sigma),\tilde M(\func))}{
|1 - C_1(\sigma,\delta^{*}(\sigma),\tilde M(\func))|
} \max{\{1,C_{1}(\sigma,\delta^{*}(\sigma),\tilde{M}(\func))
\}}^N
\eta^{-\mu_1}.
\end{split}
\label{erroranalysis:sparsegrid:estimate2}
\end{equation}
\label{erroranalysis:theorem1}
\end{theo}
\begin{proof} This is proved by applying the inequality of equation \eqref{errorestimates:eqn4}
to the proof of Theorems 3.10 and 3.11 in
  \cite{nobile2008a}.
  \end{proof}

In many practical cases not all dimensions of $\Gamma$ are equally important. 
In these cases the dimensionality of the sparse grid
can be significantly reduced by means of anisotropic sparse grids. It
is straightforward to build related anisotropic sparse approximation
formulas by having the function $g$ act differently on different input
random variables $q_n$. Anisotropic sparse stochastic collocation
\cite{nobile2008b} combines the advantages of isotropic sparse
collocation with those of anisotropic full tensor product collocation.


\section{Analyticity of the Newton iteration}
\label{Analyticity}

It is profitable here for purposes of clarification to consider to
consider the Newton iteration in a general function space context.
Specifically, let $X$ and $Y$ (see section \ref{background}) be
Banach spaces.  Consider the following problem: Find $\bx \in D
\subset X$ such that
\begin{equation}
\bbf(\bx,\bqq) = \0,
\label{analytic:eqn0}
\end{equation}
where $\bqq \in \Gamma$ and $\bbf:D\times\Gamma \rightarrow Y $.
Equation \eqref{analytic:eqn0} is then solved using the Newton
iteration under the conditions of the Newton-Kantorovich Theorem (see
section \ref{background}).


The convergence rate of the Newton iterates based on a sparse grid
approximation as a function of grid size is directly affected by the
regularity properties with respect to parameters $\bqq \in \Gamma$.
Regularity is characterized in terms of an analytic extension in
$\bbC^{N}$ of the iterates.


In the sequel we will treat $\bqq\in \Gamma$ as a random parameter,
which will be suppressed occasionally. Thus for example, below we will
write $\bbf \equiv \bbf_\bqq$, $\bJ \equiv \bJ_\bqq$, $M_v \equiv
M_{v,\bqq}$, etc., and we can write $\bff(\bx,\bqq)\equiv
\bff_\bqq(\bx)$.  For any $\bqq \in \Gamma$ (which we fix for now)
consider the Newton sequence
\begin{equation}
  \bx^{v} = M_{v}(\bx^{v-1}) \equiv \bx^{v-1} -
  \bJ(\bx^{v-1})^{-1}\bbf(\bx^{v-1}).
\label{analytic:eqn1}
\end{equation}
where $\bJ: X \rightarrow Y $ is the Fr\'echet derivative of $\bbf:D
\rightarrow Y$.  Assume that $\bx_0 \in D_0 \subset D$ and for all $v
\in \bbN$ let $D_v \subset D$ be the successive images under the map
$M_{v}$, so that $M_v(D_{v-1})=D_v$. Note at this point the random
parameter $\bqq\in\Gamma$ is unchanging throughout the iteration; the
iterated domains $D_v$ however depend on $\bqq$.

Suppose that the parameter $\bqq$ is now extended to a complex
parameter $\bg$ with $\bg \in \Psi \supset \Gamma$, where $\Psi\subset
\bbC^N$.  We can now form a complex extension of the sequence
\eqref{analytic:eqn1} as follows.  We will complexify the pair
$(\bx,\bqq)$ into a pair of complex variables $(\bz,\bg)$, with $\bz$
the complexification of $\bx$.  Assume that for $\bbf(\bx^{0}) \equiv
\bbf_\bqq(\bx^0):D_0 \rightarrow E_{0}$ there exists an analytic
extension $\bbf(\bx^{0}) \equiv \bbf_{\bg} (\bz^{0}): \Theta_0
\rightarrow \Phi_0$, where $\Theta_0$ and $\Phi_0$ are contained in a
suitable complex Banach spaces, which are the respective complex
extensions of $X$ and $Y$. Given a function $\bff$ on a real linear
domain $D$, and a function $\bff^*$ on a complex linear domain
$\Theta\supset D$, we say that $\bff^*$ is an {\it analytic extension}
of $\bff$ if $\bff^*$ is analytic on its domain, and the restriction
$\bff^*|_D=\bff$.  When there exists an analytic extension $\bff^*$ we
say that $\bff$ can be {\it analytically extended}.

\begin{rem}
Note that as before we write $\bbf \equiv \bbf_\bg$, $\bJ \equiv
\bJ_\bg$, $M_v \equiv M_{v,\bg}$, etc.  It is understood from context
that the notational equivalence is over the extension of the variables
$(\bx,\bqq)$ into the complex pair $(\bz,\bg)$.
\end{rem}

Similarly, assume that $\bJ(\bx^{0}) \in L(D_0,E_0)$ can be extended
analytically as $\bJ(\bz^{0}) \in L(\Theta_0,\Phi_0)$.  Here $L(\cdot,
\cdot)$ is the space of bounded operators between two spaces. Through
the above complexifications, equation (12) defines a complexification
of the mapping $M_v$.  We now repeat the above iteration using the
complexified maps defined here.  Thus there exists a series of sets
$\Theta_0,\dots,\Theta_v$ and $\Phi_0,\dots,\Phi_{v}$, such that for
$\bbf(\bx^{v}):D_{v} \rightarrow E_{v}$ and $\bJ(\bx^{v}) \in
L(D_{v},E_{v})$ the analytic extensions $\bbf(\bz^{v}):\Theta_{v}
\rightarrow \Phi_{v}$ and $\bJ(\bz^{v}) \in L(\Theta_{v},\Phi_{v})$
are onto.  Thus the sequence \eqref{analytic:eqn1} is extended in
$\bbC^{N}$ as follows: Let $\bz^0=\bx^0$ and for all $v \in \bbN$ and
$\bg \in \Psi$ (which we also fix) form the sequence
\begin{equation}
  \bz^{v} =
M_{v}(\bz^{v-1}) \equiv  
  \bz^{v-1} - \bJ(\bz^{v-1})^{-1}\bbf(\bz^{v-1}),
\label{analytic:eqn1a}
\end{equation}
where $M_{v}:\Theta_{v-1}
\rightarrow \Theta_{v}$.

\begin{rem}
  The domain $\Theta_0$ contains the initial condition $\bz^{0}$.
  Under certain assumptions and with a judicious choice of $\Theta_0$
  it can be shown that the sequence in \eqref{analytic:eqn1a}
  converges in a pointwise sense inside $\Theta_0$. This will be
  explored in detail in section \ref{regions}.
\end{rem}

Suppose that $\bz^{v}$ is an analytic extension of $\bx_{v}$ on $\Psi
\subset \bbC^{N}$ (see Figure \ref{analyticy:extensionfigure}).  Then
the convergence rates of the sparse grid applied to any entry of
interest of $\bz^{v}$ can be characterized. The size of the set $\Psi$
determines the regularity properties of the solution. From the sparse
grid discussion in section \ref{sparsegrids} we embed a polyellipse
${\mathcal E}_{\hat \sigma_1, \dots, \hat \sigma_{N}} : =
\Pi_{n=1}^{N}{\mathcal E}_{n,\hat \sigma_n}$ in $\Psi$. From Theorem
\ref{erroranalysis:sparsegrid:estimate} the
$L^{\infty}_{\rho}(\Gamma)$ convergence rate of the sparse grid is
sub-exponential (or algebraic) with respect to the number of sparse
grid knots $\eta$. The decay of the sparse grid is dominated by
$\sigma = \min_{n=1,\dots,N} \sigma_n$. Thus, the larger $\sigma$ is
the faster the convergence rate.


\begin{figure}[htp]
\begin{center}
\begin{tikzpicture}
    \begin{scope}[scale = 1.5]

      \node[scale = 1.5,
        shape=semicircle,rotate=270,fill=blue!40,
      semitransparent,inner sep=12.7pt, anchor=south, outer sep=0pt]
    at (1,0) (char) {};
    \node[scale=1.5, 
      shape=semicircle,rotate=90,fill=blue!40,
      semitransparent,inner sep=12.7pt, anchor=south, outer sep=0pt]
    at (-1,0) (char) {}; \path
    [fill=blue!40,semitransparent]
    (-1.001,-1) rectangle (1.001,1.001);


    \draw [line width=2] (-1, 0) -- (1, 0) node [below left]  {$\Gamma $};
    
    \node at (0.2,-0.2) {$\bqq$};

    \draw [line width=2,dashed,->, >=latex]
    (0, 0) -- (-0.6, 0.95) node [above]  {$\bv$};

    \draw [->] (-2.5, 0) -- (2.5, 0) node [above left]  {$\R^{N} $};
    \draw [->] (0,-1.5) -- (0,1.5) node [below right] {$ i \R^{N}$};
    \draw (1,-3pt) -- (1,3pt)   node [above] {$$};
    \draw (-1,-3pt) -- (-1,3pt) node [above] {$$};

    \end{scope}
    
    \node [below right] at (1.50,1.25) {\Large{$\Psi$}}; 
\end{tikzpicture}
\end{center}
\caption{Analytic extension of the domain $\Gamma$. Any vector $\bqq
  \in \Gamma$ is extended in $\Psi$ by adding a vector $\bv \in
  \bbC^{N}$ i.e. $\bg = \bqq + \bv$.} 
\label{analyticy:extensionfigure}
\end{figure}
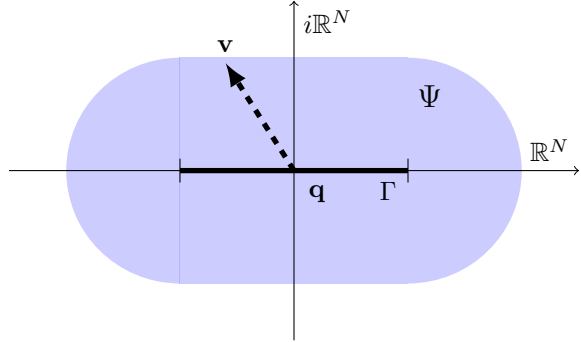

\begin{rem} For finite dimensional spaces $X=Y=\R^{m}$, $m \in \bbN_0$, 
  the Fr\'{e}chet derivative $\bJ$ corresponds
  to the Jacobian of $\bbf$. In the rest of the paper it is assumed
  that $D \subset \R^{m}$ and $\|\cdot\|$ corresponds to the standard
  Euclidean norm or the standard matrix norm, depending on context.
  For the case of the power flow equations $m$ will be
  simply related to the number of nodes of the power system
  \cite{Bergen2000}.  We will be using the notion of analytic
  extensions, which can be defined as follows.
\end{rem}

We can now prove an important theorem for our purposes. First, denote
$\bbf(\bz^v):\Theta_{v} \rightarrow \Phi_{v}$ as $\bbf^v$, and
$\bJ(\bz^v) \in L(\Theta_{v},\Phi_{v})$ as $\bJ^v$ (note that this
depends on the generic initial $\bz_0 \in \Theta_0$ at which the
Jacobian is computed).  



\begin{theo}
Assume that for all $v \in \bbN_0$: $D_v \subset X$  and 
\begin{enumerate}[(i)]
    \item 
$\bbf^v:D_{v} \rightarrow E_{v}$ can be analytically extended to
      $\bbf^v:\Theta_{v} \rightarrow \Phi_{v}$.

\item There exists a coefficient $c_v > 0$ such that
\[
  \sigma_{min}\left(
\sJ{v}
\right)
\geq c_v,
\]
where $\sigma_{min}(\cdot)$ refers to the minimum singular value,
$\bJ^v_R := \Real{\bJ^v}$ and $\bJ^v_I := \Imag{\bJ^v}$.
\end{enumerate}
Then for all $v \in \bbN_0$ there exists an analytic extension of
$\bx^{v}$ on $\Psi$.
\label{analytic:theorem1}
\end{theo}
\begin{proof}
The main strategy for this proof is to use the Cauchy-Riemann
equations. This avoids having to explicitly show that the inverse of the
complex Jacobian matrix $\bJ^v$ is analytic. The existence of an
analytic extension for $\bz^{v}$ for each separate complex dimension
is shown. The Hartog's Theorem is then used to show analyticity with
respect to all the complex dimensions.

For fixed $v$ consider the extension $\bx^{v} \rightarrow \bz^{v} =
\bx^{v} + \bw^v$ in $\Theta_v$, where $\bw^v \in \bbC^{m}$ and
$\bz^{v} \in \Theta_v$.  In complex form $\bz^{v} = \bz^{v}_{R} + i
\bz^{v}_{I}$, where $\bz^{v}_{R} = \Real\bz^{v}$, and $\bz^{v}_{I} =
\Imag\bz^{v}$. Furthermore, consider the extension of $\bqq
\rightarrow \bg = \bqq + \bv$ in $\Psi$, where $\bv \in \bbC^{N}$ and
$\bg \in \Psi$. The extension of the iteration \eqref{analytic:eqn1} on $\Theta_v
\times \Psi$ leads to the following block form iteration
\begin{equation}
  \sJ{v}
  \left(
  \szv{v+1}
  -
\szv{v}
\right)
=
-\sfv{v},
\label{analytic:eqn2}
\end{equation}
where $\bbf^v_R := \Real{\bbf^v}$ and $\bbf^v_I := \Imag{\bbf^v}$.
From $(ii)$ it follows that equation \eqref{analytic:eqn2} is well
posed and is a valid extension of equation \eqref{analytic:eqn1} on
$\Theta_v \times \Psi$. We now show that $\bz^{v+1}$ is an analytic
extension on $\Theta_v \times \Psi$.

We focus our attention on the $k^{th}$ variable of $\bz^v$ as
$z^{v}_{k}$ and write it in complex form as $z^{v}_{k} = s + iw$.  By
differentiating equation \eqref{analytic:eqn2} with respect to $s$ and
$w$ we obtain
\begin{equation}
  \begin{split}
\partial_s
\sJ{v}
\left(
\szv{v+1}
-
\szv{v}
\right)
+
\sJ{v}
\partial_s
\left(
\szv{v+1}
-
\szv{v}
\right)
&=
- \partial_s \sfv{v} \\
\partial_w
\sJ{v}
\left(
\szv{v+1}
-
\szv{v}
\right)
+
\sJ{v}
\partial_w
\left(
\szv{v+1}
-
\szv{v}
\right)
&=
-\partial_w \sfv{v}
  \end{split}
  .
\label{analytic:eqn4}
\end{equation}
From assumption $(ii)$ we conclude that $\partial_{s}\bz^{v+1}_{R}$,
$\partial_{s}\bz^{v+1}_{I}$, $\partial_{w}\bz^{v+1}_{R}$ and
$\partial_{w}\bz^{v+1}_{I}$ exist on $\Theta_v \times \Psi$. The
following step is to show that the Cauchy-Riemann equations
for $\bz^{v+1}$ are
satisfied on $\Theta_{v} \times \Psi$.

Let $P(\bz^v) := \partial_{s}\bz^{v}_{R} - \partial_{w}\bz^{v}_{I}$
and $Q(\bz^v) := \partial_{w}\bz^{v}_{R} + \partial_{s}\bz^{v}_{I}$,
then from equation \eqref{analytic:eqn4}
\[
\begin{split}
&
\begin{bmatrix}
(\partial_s \bJ^v_R - \partial_w \bJ^v_I) &
-
(\partial_s \bJ^v_I + \partial_w \bJ^v_R) 
\\
(\partial_s \bJ^v_I + \partial_w \bJ^v_R) &
-
(\partial_s \bJ^v_R - \partial_w \bJ^v_I)
\end{bmatrix}
\left(
\szv{v+1}
-
\szv{v}
\right) \\
&+ 
\sJ{v}
\left(
\begin{bmatrix}
  P(\bz^{v+1}) \\
  Q(\bz^{v+1})
\end{bmatrix}
-
\begin{bmatrix}
  P(\bz^{v}) \\
  Q(\bz^{v})
\end{bmatrix}
\right)
=
-\begin{bmatrix}
  \partial_s \bbf^v_R - \partial_w \bbf^v_I \\
  \partial_s \bbf^v_I + \partial_w \bbf^v_R \\
\end{bmatrix}.
\end{split}
\]

Now, $(i)$ implies that
$\bJ^v \in L(D_{v},E_{v})$ can be analytically extended to
$\bJ^v \in L(\Theta_{v},\Phi_{v})$.
Since $\bJ(\bz^v,\bg)$ and $\bbf^{v}(\bz^v,\bg)$ are analytic on
$\Theta_{v} \times \Psi$ then from the Cauchy-Riemann equations
\[
\begin{bmatrix}
(\partial_s \bJ^v_R - \partial_w \bJ^v_I) &
-
(\partial_s \bJ^v_I + \partial_w \bJ^v_R) 
\\
(\partial_s \bJ^v_I + \partial_w \bJ^v_R) &
-
(\partial_s \bJ^v_
R + \partial_w \bJ^v_I)
\end{bmatrix} = \0
\,\,\,\mbox{and}\,\,\,
\begin{bmatrix}
  \partial_s \bbf^v_R - \partial_w \bbf^v_I \\
  \partial_s \bbf^v_I + \partial_w \bbf^v_R \\
\end{bmatrix}
=
\0.
\]
Since $z^v_k$ is a linear polynomial of $s + iw$ then $P(\bz^{v}) =
Q(\bz^{v})=\0$ on $\bbC^{N}$ and thus $P(\bz^{v+1}) = Q(\bz^{v+1})=\0$
on $\Theta_{v} \times \Psi$. We conclude that $\bz^{v+1}$ is analytic
for the $k^{th}$ variable for all $\bz^{v} \in \Theta_{v}$ and $\bg
\in \Psi$.  Following a similar argument we can show that for $l = 1,
\dots, N$ the $l^{th}$ variable extension of $\bqq$ has leads to an analytic extension of $\bz^{v+1}$ whenever $\bz^{v} \in \Theta_v$ and $\bg \in \Psi$.  We now extend the
analyticity of $\bz^{v+1}$ on all of $\Theta_{v} \times \Psi$.

Since $z^{v+1}_k$ is analytic for all $k = 1, \dots, m$ and the
$l^{th}$ variable of $\bqq$ has an analytic extension for all $l = 1,
\dots, N$ whenever $\bz^{v} \in \Theta_v$ and $\bg \in \Psi$, then
from Hartog's theorem we conclude that $\bz^{v+1}$ is continuous on
$\Theta_{v} \times \Psi$. From Osgood's lemma it follows that
$\bz^{v+1}$ is analytic on $\Theta_{v} \times \Psi$.  From an
induction argument and using that fact that the composition of
analytic functions is analytic then it follows that $\bz^{v+1}$ is
analytic in $\Psi$, for all $v\in \bbN$.
\end{proof}

If the assumptions of Theorem \ref{analytic:theorem1} are satisfied
then $\bz^v$ is complex analytic in $\Psi$ and it is reasonable to
construct a series of sparse grid surrogate models of the entries of
the vector $\bx^v$. Note that in practice we restrict out attention to
a subset of the variables of interest of $\bx^v$. With a slight abuse
of notation denote ${\mathcal S}^{m,g}_{w}[\bx^v(\bqq)]$ as the sparse
grid approximation of the entries of interest of the vector $\bx^v$.

From Theorem \ref{erroranalysis:theorem1} we observe that the accuracy
of the sparse grid approximation is a function of i) the size of the
polyellipse ${\mathcal E}_{\hat \sigma_1, \dots, \hat \sigma_N} \subset
\Psi$ and ii)

\[
\tilde{M}(\bz^v) = \sup_{\bz^v \in \Theta_{v}
, 
    k =1,\dots,m} |z^v_k(\bg)|.
\]
If the complex sequence \eqref{analytic:eqn1a} does not converge, then
the size of the sets $\Theta_v$ can become unbounded.  In particular,
it is possible that $\tilde M(\bz^v) \rightarrow \infty$ as $v
\rightarrow \infty$ even if $\bJ(\bz^v)^{-1} \in
L(\Phi_{v},\Theta_{v})$ exists for all $v \in \bbN_0$.  Our objective
now is to analyze under what conditions the complex sequence remains
bounded. In particular, for all $v \in \bbN_{0}$, we ask if it is possible to construct
bounded regions $U \subset \bbC^{m}$ and $\Psi \subset \bbC^{N}$ such
that $\bz^v$ is contained in $U$ and thus
\[
\tilde M(\bz^v) \leq \sup_{\bz^v \in U} \|\bz^v\|_{\infty}.
\]
To help answer this question we first show that the complex sequence
\eqref{analytic:eqn1a} is itself a Newton sequence.

\begin{rem} We have to clarify what we mean by 
the Fr\'echet derivative of the complex function $\bff:\Theta_v
\rightarrow \Psi_v$.  The algebraic problem of equation
\eqref{analytic:eqn0} can be complexified as follows: Find $\bz \in
\Theta_{0}$ such that $\bff(\bz, \bg)=\0$ for all $ \bg \in
\Psi$. This can be re-written in vector form as: Find $\bz = \bz_R
+ i\bz_I \in \Theta_0$ such that $\Real \bff(\bz_R,\bz_I, \bg_R,
\bg_I)=\0$ and $\Imag \bff(\bz_R,\bz_I, \bg_R, \bg_I)=\0$ for all $\bg
= \bg_R + i\bg_I \in \Psi$. The corresponding Newton iteration is based on
\begin{equation}
  \begin{bmatrix}
  \partial_{\bz_R} \bff^v_R
  & \partial_{\bz_I} \bff^v_R
  \\
  \partial_{\bz_R} \bff^v_I
  & \partial_{\bz_I} \bff^v_I
  \end{bmatrix}
  \left(
  \szv{v+1}
  -
\szv{v}
\right)
=
-\sfv{v},
\end{equation}
where $\partial_{\bz_R} \bbf^v_R$ is the Fr\'echet derivative of
$\bbf^v_R$ with respect to the variables $\bz_R$ and similarly for the
rest. We refer to the matrix
\[
\bJ_{\bz^v} :=
  \begin{bmatrix}
  \partial_{\bz_R} \bbf^v_R
  & \partial_{\bz_I} \bbf^v_R
  \\
  \partial_{\bz_R} \bbf^v_I
  & \partial_{\bz_I} \bbf^v_I
  \end{bmatrix}
\]
as the Fr\'{e}chet derivative of $\bff:\Theta_v \rightarrow \Psi_v$.
\end{rem}

\begin{theo}
Suppose assumptions i) and ii) of Theorem \ref{analytic:theorem1} are
satisfied. Then the complex analytic extension $\bJ(\bz^v) \in
L(\Theta_{v},\Phi_{v})$ of $\bJ(\bx^v) \in L(D_{v},E_{v})$ is
equivalent to the Fr\'echet derivative of $\bbf^v:\Theta_{v}
\rightarrow \Phi_{v}$, i.e.
\[
\sJ{v} = 
  \begin{bmatrix}
  \partial_{\bz_R} \bbf^v_R
  & \partial_{\bz_I} \bbf^v_R
  \\
  \partial_{\bz_R} \bbf^v_I
  & \partial_{\bz_I} \bbf^v_I
  \end{bmatrix}.
\]
\label{analytic:lemma1}
\end{theo}
\begin{proof}
We first prove this result for $m = 1$ dimension. Suppose that $f:D
\rightarrow \R$, is a Fr\'echet differentiable function and let $f:\Xi
\rightarrow \bbC$ be the analytic continuation on the non-empty open
set $\Xi \subset \bbC$. The analytic function $f:\Xi \rightarrow \bbC$
can be rewritten as $f(x,y) = f_{R}(x,y) + if_{I}(x,y)$ for all $x +
iy \in \Xi$. Since $f$ is analytic on $\Xi$, from the the identity
theorem \cite{Ablowitz2003} (uniqueness of complex analytic
extensions) we have that $f_{R}$ and $f_{I}$ are unique in
$\Xi$. Furthermore, since $f$ is analytic the Cauchy-Riemann equations
are satisfied. Thus
\begin{equation}
\begin{bmatrix}
  \partial_x f_R & -\partial_x f_I \\
  \partial_x f_I & \partial_x f_R
\end{bmatrix} 
=
\begin{bmatrix}
  \partial_x f_R & \partial_y f_R \\
  \partial_y f_I & \partial_y f_I
\end{bmatrix}     
\label{analytic:lemma1:eqn1}
\end{equation}
in $\Xi$ and $f:\Xi \rightarrow \bbC$ is Fr\'echet differentiable.
Now, $\partial_x f(x,y) = \partial_x f_R(x,y) + i \partial_x f_I(x,y)$
in $\Xi$, and from the uniqueness property of the Fr\'echet derivative
all the terms are unique. Recall that $\partial_x f(x)$ defined in $D$
is the Fr\'echet derivative of $f:D \rightarrow \R$.  Write the
analytic extension of $\partial_x f(x)$ (defined in $D$) on $\Xi$ as
$g(x,y) + ih(x,y)$, with
$x+iy \in \Xi$.  Since $\partial_x f(x) = \partial_x f(x,y) =
\partial_x f_R(x,y) + i \partial_x f_I(x,y)$ for $y = 0$ and $x \in
D$, from the uniqueness of the analytic extension we conclude
$g = \partial_x f_R$ and $h = \partial_x f_I$ for all $x + iy \in
\Xi$.  From equation \eqref{analytic:lemma1:eqn1} the conclusion follows.

We can now prove our statement for the general case using a simple extension of the
above argument. Since $\bbf^v:\Theta_{v} \rightarrow \Phi_{v}$ is
complex analytic, from the identity theorem \cite{Ablowitz2003} it
is the unique extension of $\bbf^v:D_{v} \rightarrow E_{v}$.  (Note
that the unique extension of the identity theorem applies in
multi-variate case, which includes the variables $\bx^v$ and $\bqq$ in
the domains $\Theta_v$ and $\Psi$ respectively.)  From the
Cauchy-Riemann equations the functions $\bbf^v:\Theta_{v} \rightarrow \Phi_{v}$ are
Fr\'echet differentiable and unique. Now, with a slight abuse of
notation, denote $\bJ_{\bz^v_R}$ as the Jacobian of $\bbf^v:\Theta_{v}
\rightarrow \Phi_{v}$, with respect to the real variables $\bz^v_R$
only. By using the above one dimensional argument we can show that the
analytic extension of each entry of $\bJ(\bx^v)$ matches $\bJ_{\bz^v}$ on the real part of $\Theta_v$.  From the Cauchy-Riemann
equations we conclude that $\bJ(\bz^v) \in L(\Theta_{v},\Phi_{v})$ is
equivalent to the Fr\'echet derivative of $\bbf^v:\Theta_{v}
\rightarrow \Phi_{v}$.
\end{proof}

From Theorem \ref{analytic:lemma1} it follows that the complex sequence \eqref{analytic:eqn1a} is a Newton sequence. We can
now apply the Newton-Kantorovich Theorem to study the sequence
convergence as $v \rightarrow \infty$.

\subsection{Regions of Analyticity}
\label{regions}

The size of a polyellipse embedded in the domain $\Psi$ and the
magnitude of $\bz^{v} \in \Theta_{v}$ (for any $v \in \bbN_{0}$)
directly impacts the accuracy of the sparse grid (c.f. Theorem
\ref{erroranalysis:theorem1}). For each $v \in \bbN_{0}$ the size of
the domains $\Theta_v$ and $\Psi$ will be characterized by the
magnitude of the minimum singular value
\begin{equation}
\sigma_{min}\left( \sJ{v} \right) \geq c_v > 0
\label{regions:eqn1}
\end{equation}
for some $c_v > 0$.  However, constructing the domain $\Psi$ would
require imposing inequality conditions for each Newton iteration. This
leads to a highly complex coupled problem that is hard to
solve. Moreover, if the complex extension $\bz^v$ grows rapidly with
respect to $v$ then the size of the domain $\Psi$ will be most likely
severely constrained. In contrast, by applying the Newton-Kantorovich
Theorem it is sufficient to impose conditions on the initial Jacobian
$(v = 0)$ to construct a region of analyticity for $\Psi \subset
\bbC^{N}$. Furthermore, the size of the iteration $\bz^v$ will be
controlled.

Consider the iteration
\begin{equation}
\balpha^{v+1} = \balpha^{v} -
\bJ(\balpha^v,\bg)^{-1} \bbf(\balpha^{v},\bg),
\label{regions:eqn1a}
\end{equation}
where $\bg \in \Psi$,
\[
\balpha^0 := \begin{bmatrix}
  \bx_0 \\
  \0
  \end{bmatrix}
,\,
\balpha^v := \szv{v},\,
  \bJ(\balpha^v,\bg) := \sJ{v},
  \,\mbox{and}\,
\bbf(\balpha^v,\bg) 
:=
\sfv{v},
\]
for all $v \in \bbN_{0}$.
\begin{rem}
From Theorem \ref{analytic:lemma1} or, alternatively, the
Cauchy-Riemann equations, the matrix $\bJ(\balpha^v,\bg)$ corresponds
to the Fr\'echet derivative of $\bbf(\balpha^v,\bg)$.  Thus the
sequence \eqref{regions:eqn1a} is an Newton iteration and the
Newton-Kantorovich Theorem can be used to analyze its convergence
properties.
\end{rem}

\begin{asum}
For all $\bqq \in \Gamma$ Assumption \ref{regions:condition1} is satisfied.
\label{regions:condition2}
\end{asum}

\begin{asum}
Assume that $\tilde D$, where $D \subset \tilde D$, is an open convex
set in $\R^{2 m}$ and the following Lipschitz condition is satisfied:
\[
  \| \bJ(\bx,\bg) - \bJ(\by,\bg) \| \leq \lambda_{e} \| \bx - \by \|,
\]
for all $\bx,\by \in \tilde D$, $\bg \in \Psi$, and $\lambda_{e} \geq
0$. Furthermore assume that for all $\bg \in \Psi$
\[
  \| \bJ(\balpha^0,\bg)^{-1} \| \leq \varkappa_{e},
  \| \bJ(\balpha^0,\bg)^{-1} \bbf(\balpha^0,\bg)\| \leq \delta_{e}
  , h_{e}
  = 2 \varkappa_{e} \lambda_{e} \delta_{e} \leq 1,
\]
  and $U(\balpha^{0},t^{*}_e) \subset \tilde D$, where $t^{*}_{e} =
  \frac{2}{h_{e}}( 1 - \sqrt{1 - h_{e}})\delta_{e}$.
  \label{regions:condition3}
\end{asum}
\begin{theo}
If Assumptions \ref{regions:condition2}  and \ref{regions:condition3} are 
satisfied then for all $\bg \in \Psi$
\begin{enumerate}
\item The Newton iterates $\balpha^{v+1} = \balpha^{v} +
  \bJ(\balpha^v,\bg)^{-1} \bbf(\balpha^{v},\bg)$ exist and $\balpha^v
  \in U(\balpha_0,t^{*}_{e}) \subset \tilde D$.
\item $\balpha^* := \lim_{v \rightarrow \infty} \balpha^v$ exists,
  $\balpha^* \in \overline{U(\balpha^0,t^{*}_{e})}$, and
  $\bbf(\balpha^{*},\bg) = \0$.
\item $\bJ(\balpha^v,\bg)^{-1}$ exists for all $v \in \bbN$ and
  equation \eqref{regions:eqn1} is satisfied.
\end{enumerate}
\label{regions:theorem3}
\end{theo}
\begin{proof}
Immediate application of the Newton-Kantorovich Theorem.
\end{proof}

\begin{rem}
  From Assumptions \ref{regions:condition2} and
  \ref{regions:condition3} and from the fact that extended Newton
  iteration is a valid extension of the sequence \eqref{analytic:eqn1}
  then we have that $\lambda_{e} \geq \lambda$, $\varkappa_{e} \geq
  \varkappa$, $\delta_{e} \geq \delta$, $h_{e} \geq h$.  This
  implies that $t^{*}_{e} \geq t^{*}$ and therefore
  $U(\bx_{0},t^{*}) \subseteq U(\balpha_{0},t^{*}_{e})$ 
  for all $\bg \in \Psi$ (See Figure \ref{newtoniteration:extensionfigure}). 
  From the Newton-Kantorovich Theorem it follows that that 
  $\balpha^{v} \in U(\balpha_{0},t^{*}_{e})$ for all $v \in \bbN_0$.
\end{rem}

\begin{figure}[htp]
\begin{center}
\begin{tikzpicture}[scale = 2]
    \begin{scope}

    \filldraw[fill=blue!40, semitransparent]
    (0,0) ellipse (2 and 1);


    \filldraw[fill=blue!80, semitransparent, dashed]
    (0,0) circle (1);

    \draw [line width=2] (-1.5, 0) -- (1.5, 0) node [below left]  {$D$};
    
    \node at (0.2,-0.2) {$\bx_{0}$};

    \node at (-0.40,-0.4) {$U(\balpha_{0},t^{*}_{e})$};

    \node at (0.5,0.2) {$t^{*}$};
    \node at (-0.5,0.2) {$-t^{*}$};

    \draw [line width=1,dashed,->, >=latex]
      (0, 0) -- (-0.48, 0.87) node [above left]  {$t^{*}_e$};

    \draw [->] (-2.5, 0) -- (2.5, 0) node [above left]  {$\R^{m}$};
    \draw [->] (0,-1.5) -- (0,1.5) node [below right] {$ \R^{m}$};
    \draw (0.5,-3pt) -- (0.5,3pt)   node [above] {$$};
    \draw (-0.5,-3pt) -- (-0.5,3pt) node [above] {$$};
    \end{scope}
    
    \node [below right] at (1.50,1.25) {\Large{$\tilde D$}}; 
\end{tikzpicture}
\end{center}
\caption{Region of convergence $U(\balpha_0,t^{*}_e)$ for the extended
  Newton iteration.}
\label{newtoniteration:extensionfigure}
\end{figure}
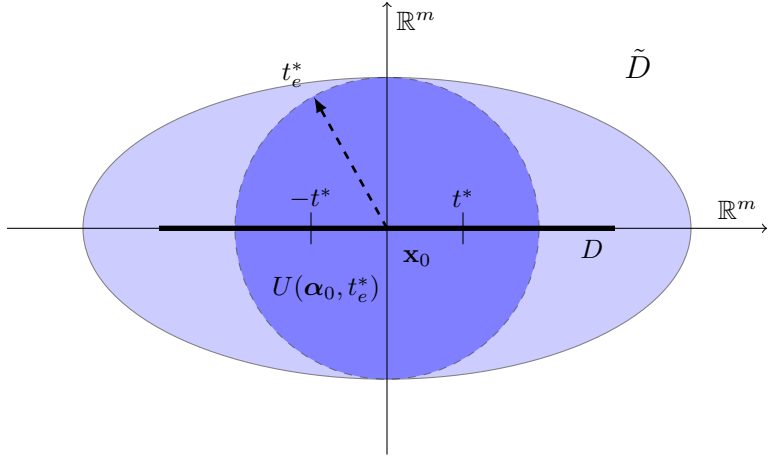

We can construct a region $\Psi$ such that for all $\bg \in \Psi$ the
extended Newton iteration converges. 
Let $\by =
\begin{bmatrix}
  \by_R
  \\
  \by_I
\end{bmatrix}
=
\begin{bmatrix} \Real{\bg}
  \\
  \Imag{\bg}
\end{bmatrix}
$, and apply the multivariate Taylor theorem for each $k=1,\dots,n$,
$l= 1,\dots,n$ entry of the Jacobian matrix
$\bJ_R(\balpha_{0},\bg)$. Evaluating $\bg$ at $\bqq + \bv$, we have
that
\[
  [\bJ_R(\balpha_0, \bqq + \bv_R, \0 + \bv_I) ]^{k,l}=
  [\bJ_R(\bx_{0},\bqq,\0) ]^{k,l} + \bR_{k,l}
  (\bx_0) ]
\begin{bmatrix}
\bv_{R} \\
\bv_{I} 
\end{bmatrix},
\]
$\bv_R := \Real{\bv}$, $\bv_I := \Imag{\bv}$, and
\[
\bR_{k,l}(\bx_{0}) = [
  R^{1}_{k,l}(\bx_{0}), \dots,
  R^{2m}_{k,l}(\bx_{0})
].
\]
The entries of the remainder term $\bR_{k,l}(\bx_{0})$ are bounded by
\[
|R^{\beta}_{k,l}(\bx_{0}) | \leq \max_{t \in (0,1)} \left|
\partial_{y_{\beta}} \left[\bJ_{R}
  \left(\bx_{0},
  \left[
    \begin{array}{c}
  \bqq \\
  \0
  \end{array}
  \right]
+ t
\left[\begin{array}{c}
  \bv_{R} \\
  \bv_{I}
  \end{array}
  \right]\right)
\right]^{k,l}\right|,
\]
where $\partial_{y_{\beta}}$ refers to the derivative of the
$\beta^{th}$ variable of the vector $\by$. Form the matrix
\[
\bE  :=
\begin{bmatrix*}
\0 &
-\bJ_I(\bx_{0},\bqq,\bv_R,\bv_I) \\
\bJ_I(\bx_{0},\bqq,\bv_R,\bv_I) & \0 \\
\end{bmatrix*}
+
\begin{bmatrix*}[c]
\bQ(\bx_{0},\bqq,\bv_R,\bv_I) & \0 \\
\0  & \bQ(\bx_{0},\bqq,\bv_R,\bv_I) \\
\end{bmatrix*}
\]
and let
\[
\bQ_{k,l}(\bx_{0},\bqq,\bv_R,\bv_I)
= \bR
_{k,l}(\bx_{0})
\begin{bmatrix}
\bv_{R} \\
\bv_{I} 
\end{bmatrix}
\]
be the $k= 1,\dots,n$, $l = 1,\dots,n$ entry of the matrix $\bQ$.  Then
\[
\begin{bmatrix*}[c]
  \bJ_R(\bx_0,\bqq)
  & -\bJ_I(\bx_{0},\bqq,\bv_R,\bv_I) \\ \bJ_{I}(\bx_{0},\bqq,\bv_R,\bv_I)
  & \bJ_R(\bx_0,\bqq) \\
\end{bmatrix*}
=
\bbJ
+ \bE
=
\bbJ(\bI + \bbJ^{-1}\bE),
\]
where $\bbJ :=
\begin{bmatrix*}[c]
  \bJ_R(\bx_0,\bqq)    & \0  \\
  \0  & \bJ_R(\bx_0,\bqq)  \\
\end{bmatrix*}$.
\begin{theo}
Suppose that $\varkappa_e \geq \varkappa$ and 
  \[
  \| \bE(\balpha^0,\bg) \| <
  \frac{1 -
    \frac{\varkappa}{\varkappa_e}
  }{\varkappa}
\]
whenever $\bg \in \Psi$ then
\[
\| \bJ(\balpha^0,\bg)^{-1} \| \leq \varkappa_{e}.
\]
\label{regions:theorem1}
\end{theo}
\begin{proof}
First note that
\begin{equation}
\|\bJ(\balpha_0,\bg)^{-1}\| \leq
\| (\bbJ(\bI + \bbJ^{-1}\bE))^{-1}\| \leq
\| \bbJ^{-1} \| \| (\bI + \bbJ^{-1}\bE)^{-1} \|.
\label{regions:eqn4}
\end{equation}
From Lemma 2.2.3 in \cite{Golub1996}, if $\| \bbJ^{-1} \bE \|_{2} < 1$
then $(\bI + \bbJ^{-1}\bE)$ is invertible and
\[
\|(\bI + \bbJ^{-1}\bE)^{-1}\| < \frac{1}{1 - \|\bbJ^{-1}\bE \|}.
\]
Given that $\|\bbJ(\bx_0,\bqq)^{-1}\| \leq \varkappa$ (From Assumption
\ref{regions:condition1} ) whenever $\bqq \in \Gamma$, it follows
\begin{equation}
  \| \bbJ^{-1} \| \|\bI + \bbJ^{-1}\bE\| <
  \frac{\varkappa}{1 - \|\bbJ^{-1}\bE \|}
  \,\,\,\mbox{and}\,\,\,
\| \bbJ^{-1}\bE \| \leq \|\bbJ^{-1}\| \|\bE \| \leq
\varkappa  \|\bE \|.
\label{regions:eqn5}
\end{equation}
We conclude that if
\[
\| \bE(\balpha^0,\bg) \| < \frac{1 -
    \frac{\varkappa}{\varkappa_e}
}{\varkappa}
\]
whenever $\bg \in \Psi$, then from Equations
\eqref{regions:eqn4} and \eqref{regions:eqn5}
\[
\| \bJ(\balpha^0,\bg)^{-1} \| \leq \varkappa_{e}.
\]
\end{proof}
Applying the multivariate Taylor's theorem for each $k=1,\dots,n$,
entry of the vector $\bbf_R(\balpha_{0},\bg)$ where $\bg = \bqq +
\bv$, we have 
\[
  [\bbf_R(\balpha_0, \bqq + \bv_R, \0 + \bv_I) ]^{k}=
  \left[\bbf_R(\bx_{0},\bqq,\0) ]^{k} + \bS_k
  (\bx_0,\bqq,\0) \right]
\begin{bmatrix}
\bv_{R} \\
\bv_{I} 
\end{bmatrix},
\]
where $\bv_R := \Real{\bv}$, $\bv_I := \Imag{\bv}$, and
\[
\bS_k(\bx_{0},\bqq,\0) = [
  S^{1}_{k}(\bx_{0},\bqq,\0), \dots,
  S^{2m}_{k}(\bx_{0},\bqq,\0)
].
\]
The remainder term $\bS_{k}(\bx_{0},\bqq)$ is bounded by
\[
|S^{\beta}_{k}(\bx_{0},\bqq,\0) | \leq \max_{t \in (0,1)} \left|
\partial_{y_{\beta}} \left[\bbf_{R} \left(
\begin{bmatrix}
\bqq \\
\0  
\end{bmatrix}
+
\begin{bmatrix}
\bv_{R} \\
\bv_{I} 
\end{bmatrix}
\right)
\right]^{k}\right|,
\]
where $\partial_{y_{\beta}}$ refers to the derivative of the $\beta^{th}$
variable of $\by$. We can now rewrite the vector
$\bbf(\balpha_0,\bg)$ as
\[
\bbf(\balpha_0,\bg) = \bbF(\bx_0,\bqq) + \bG(\bx_{0},\bqq,\bv_R,\bv_I),
\]
where $\bbF :=
\begin{bmatrix*}[c]
  \bbf_R(\bx_0,\bqq)  \\
  \0  \\
\end{bmatrix*}$, 
$\bG  :=
\begin{bmatrix*}
\bP(\bx_{0},\bqq,\bv_R,\bv_I) \\
\bbf_I(\bx_{0},\bqq,\bv_R,\bv_I) \\
\end{bmatrix*}$
and 
\[
\bP_{k}(\bx_{0},\bqq,\bv_R,\bv_I)
= \bS
_{k}(\bx_{0},\bqq,\0) 
\begin{bmatrix}
\bv_{R} \\
\bv_{I} 
\end{bmatrix}.
\]
\begin{theo}
  Suppose that $\varkappa_e \geq \varkappa$ and
  $\delta_e \geq \delta$. Then if
  \[
  \| \bG(\balpha^0,\bg) \| <
  \frac{\delta_e}{\varkappa_e}
  -
  \frac{\delta}{\varkappa}
\]
whenever $\bg \in \Psi$, it follows
\[
\| \bJ(\balpha_0,\bg)^{-1} \bbf(\balpha_0,\bg) \| \leq \delta_e.
\]
\label{regions:theorem2}
\end{theo}
\begin{proof}
For each of the entries $k= 1,\dots,n$ of the vector $\bP$.
\begin{equation}
\begin{split}
\| \bJ(\balpha_0,\bg)^{-1} \bbf(\balpha_0,\bg) \|
&=
\|
(\bbJ + \bE)^{-1}
(\bbF +
\bG
) \|
\leq
\|
(\bI + \bbJ^{-1}\bE)\bbJ^{-1}
(\bbF +
\bG
) \| \\
& \leq
\|
(\bI + \bbJ^{-1}\bE)\bbJ^{-1}
\bbF \| +
\|
(\bI + \bbJ^{-1}\bE)\bbJ^{-1}
\bG\| \\
& \leq
\|
\bI + \bbJ^{-1}\bE \| \|\bbJ^{-1}
\bbF \| +
\|
\bI + \bbJ^{-1}\bE \|
\|
\bbJ^{-1} \|
\|
\bG\| \\
& \leq
\|
\bI + \bbJ^{-1}\bE \| \delta +
\|
\bI + \bbJ^{-1}\bE \| \|
\bG\|\varkappa. 
\end{split}
\label{regions:eqn6}
\end{equation}
Since $\| \bE \|_{2} < \frac{1 - \frac{\varkappa}{\varkappa_e}
}{\varkappa} < \varkappa$ (from Theorem \ref{regions:theorem1}) from
Lemma 2.2.3 in \cite{Golub1996} it follows that $(\bI + \bbJ^{-1}\bE)$
is invertible and
\begin{equation}
\|(\bI + \bbJ^{-1}\bE)^{-1}\| < \frac{1}{1 - \|\bbJ^{-1}\bE \|}
\leq \frac{\varkappa_e}{\varkappa}.
\label{regions:eqn7}
\end{equation}
Combining e
quations \eqref{regions:eqn6} and \eqref{regions:eqn7} we
have \[
\begin{split}
\| \bJ(\balpha_0,\bg)^{-1} \bbf(\balpha_0,\bg) \|
&<
\frac{\varkappa_e}{\varkappa}
( \delta + \|\bG(\bx_{0},\bqq,\bv_R,\bv_I)\| \varkappa ).
\end{split}
\]
The result follows.
\end{proof}
From the values of $\varkappa,\delta,\lambda$ and
$\varkappa_e,\delta_e,\lambda_e$ and Theorems \ref{regions:theorem1} and
\ref{regions:theorem2}, the region of analyticity $\Psi$ can be
constructed.  From this region, convergence rates
from Theorem \ref{erroranalysis:theorem1} for the sparse grid
interpolation can be estimated. If we are interested in
forming a sparse grid for each of the entries of the vector $\bx^{v}
\in \R^{n}$, then there are potentially $n$ sparse grids. To estimate
the convergence rate of the sequence of sparse grids it is sufficient to embed a
polydisk $\mcE_{\sigma_1, \dots, \sigma_N} \subset \Psi$ for a
suitable set of coefficients $\{ \sigma_1 \dots, \sigma_N
\}$. Furthermore since $\balpha^{v} \in U(\alpha_0, t^*_e)$, the
maximal coefficient $\tilde M(\bz^{v})$ can be bounded as
\[
\tilde{M}(\bz^{v}) \leq t^{*}_e + \|\bx_0\|_{l^{2}(\R^{2m})}.
\]

\section{Application to power flow}
\label{Applications}

The theory developed in Section \ref{Analyticity} can be applied to
the computation of the statistics of stochastic power flow. In
particular we concentrate on the random 
perturbations of the generators, loads  and admittance uncertainty of 
the transmission lines. Much of the power system network model presented 
in this section is based on \cite{Bergen2000}.

Consider a network with $m+1$ mechanical constant power generators. The
electrical power injected into the network at each generator is given
by
\begin{equation}
P_{G_k} = \sum_{l = 0}^{m} V_kV_lsin(\theta_k - \theta_l +
\varphi_{k,l})|Y_{k,l}|,
\label{PF:eqn1}
\end{equation} where the operands of the summation are the power from bus $k$ transmitted to bus $l$ through a line with
admittance $Y_{k,l}=G_{k,l} + iB_{k,l}$, phase shift $\varphi_{k,l}$,
and voltage $V_k$ at the buses. These form the algebraic constraints
of the power system.  The dynamic constraints at generator $k$ are
given by
\begin{equation}
M_k \ddot{\theta_{i}} + D_k \dot{\theta_{i}} + P_{G_k} = P_{M_k} +
P_{I_k}(\omega) + P_{L_k}(\omega)
\label{PF:eqn2}
\end{equation}
where $M_k$ is the moment of inertia of generator $i$, $D_k$ is the
damping factor, $P_{M_k}$ denotes the mechanical power,
$P_{L_k}(\omega)$ is the stochastic load and $P_{I_k}(\omega)$ is the
intermittent stochastic power applied to bus $i$. Equations
\eqref{PF:eqn1} and \eqref{PF:eqn2} constitute the {\it swing equation
  model}. Since the intermittent power generators and loads are
stochastic, the rotor angle $\theta_k(\omega)$ and power generation
$P_{G_k}(\omega)$ will be stochastic as well. A simple example of a 3
bus power system is shown in Figure \ref{example1:fig1}.  From the
steady state response the {\it power flow equations} are given by
\[
\begin{split}
  P_k(\bx)
  &=
  \sum_{l=0}^{m} V_k V_l[G_{ik}cos(\theta_k - \theta_l) +
    B_{ik} sin(\theta_k - \theta_l)]
  \\
    Q_k(\bx)
  &=
  \sum_{l=0}^{m} V_kV_l[G_{ik}sin(\theta_k - \theta_l) +
    B_{ik} cos(\theta_k - \theta_l)]
\end{split}
\]
for $k = 0,\dots,m$.

It is assumed that at each node the active and reactive power injections
(or loads) are given by $P_1, \dots, P_{m}$ and $Q_1, \dots,
Q_{m}$. The first bus is assumed to be slack bus with known angle
$\theta_0 = 0$ and fixed voltage $V_0$.

\begin{rem}
  According to power system convention, the numbering of
  the buses (nodes) starts with 1 instead of 0. To simplify the
  notation in this section we start from 0. However, for the
  examples and numerical results we revert to the power system
  standard.
\end{rem}

In this paper we limit our discussion of power flow to the case where
the power injections $P_1, \dots, P_{m}$ and $Q_1, \dots, Q_{m}$ are
assumed to be known, but could be stochastic. The unknowns are formed
by the angles $\theta_1,\dots,\theta_{m}$ and voltages $V_1,
\dots,V_{m}$. The power flow equations are solved with a Newton
iteration and posed as
\[
\btheta :=  
\begin{bmatrix}
  \theta_1\\
  \vdots \\
  \theta_{m}
\end{bmatrix},
\bV:=
\begin{bmatrix}
  V_1 \\
  \vdots \\
  V_{m}
\end{bmatrix},
\bx:=
\begin{bmatrix}
  \btheta \\
  \bV
\end{bmatrix},
\bff(\bx) =
\begin{bmatrix}
  \Delta P(\bx)
  \\
  \Delta Q(\bx)
\end{bmatrix},
\]
where
\[
\Delta P(\bx)
:=
\left[
  \begin{array}{c}
    P_1(\bx) - P_1 \\
    \vdots \\
    P_{m}(\bx) - P_{m}
    \end{array}
  \right]
\,\,\,\mbox{and}\,\,\,
\Delta Q(\bx) :=
\left[
  \begin{array}{c}
    Q_1(\bx) - Q_{1} \\
    \vdots \\
    Q_{m}(\bx) - Q_{m}\\
    \end{array}
  \right].
\]
The Jacobian matrix is given in block form as
\[
\bJ =
\begin{bmatrix}
  \bJ_{11} & \bJ_{12}
  \\
  \bJ_{21} & \bJ_{22}
\end{bmatrix},
\]
where $\bJ_{11}, \bJ_{12}, \bJ_{21},\bJ_{22} \in \R^{n \times n}$. For
$k,l = 1, \dots, m$ let $\theta_{k,l} := \theta_{k} - \theta_{l}$ and
if $k \neq l$
\[
\begin{array}{ll}
  \bJ^{11}_{k,l} = \,\,\,\,\,V_kV_l G_{k,l} sin(\theta_{k,l})
  - B_{k,l}(cos(\theta_{k,l})),
  &\bJ^{21}_{k,l} = -V_kV_l G_{k,l} cos(\theta_{k,l})
  + B_{k,l}(sin(\theta_{k,l})),\\
  \bJ^{12}_{k,l} = \,\,\,\,\,V_kG_{k,l} cos(\theta_{k,l})
  + B_{k,l}(sin(\theta_{k,l})),
  &\bJ^{22}_{k,l} = \,\,\,\,\,V_kG_{k,l} sin(\theta_{k,l})
  - B_{k,l}(cos(\theta_{k,l})),
  \end{array}
\]
otherwise
\[
\begin{array}{ll}
  \bJ^{11}_{k,k} = -Q_k(\bx) - B_{k,k}V_k^2 &
  \bJ^{21}_{k,k} =  \,\,\,\,\,P_k(\bx) - G_{k,k}V_k^2 \\
  \bJ^{12}_{k,k} =  \,\,\,\frac{P_k(\bx)}{V_k} + G_{k,k}V_k &
  \bJ^{22}_{k,k} =  \,\,\,\frac{Q_k(\bx)}{V_k} - B_{k,k}V_k
  \end{array}.
\]
It is clear from the structure of $\bbf$ and the Jacobian $\bJ$ that
they are analytic everywhere except for $V_k = 0$, for $k =1,\dots,m$.
However, in practice the domain $\Theta_0$ is chosen such the origin
is avoided. Otherwise the analyticity assumptions of Theorem
\ref{analytic:theorem1} are not satisfied.

There are many forms of uncertainty that can be present in the solution
of the power flow equations. We concentrate on the following cases:
\begin{itemize}
\item {\bf Random generators and loads:} The power injections $P_k$ and
  $Q_k$, $k = 1,\dots,m$, will be a function of the random vector
  $\bqq \in \Gamma$:
\[
  P_{k} + iQ_{k} = P^{0}_k(1 + c_k q_k) + Q^{0}_k(1 + c_{k+1} q_{k+1}),
  \]
where $P^{0}_k$ and $Q^{0}_k$ are the nominal power loads (or generators), $q_k \in
[-1,1]$ and $c_k,c_{k+1} \in \R$.

\item {\bf Random admittances:} 
The transmission line admittances
  $Y_{k,l}$ will be functions of the random
  vector $\bqq \in \Gamma$.  Let $\mcA$ be
the set of network index tuples $(k,l)$ such that the admittance is
stochastic. Thus for all $k,l \in \mcA$ let 
 \[
 Y_{k,l} =  G_{k,l} + iB_{k,l} = G^{0}_{k,l}(1 + 
 c_{k,l,1} 
 q_{k,l,1}
 ) + B^{0}_{k,l}
 (1 + 
 c_{k,l,2}  q_{k,l,2} 
 ),
 \]  
 where $G^{0}_{k,l}$ and $B^{0}_{k,l}$ are the nominal conductance and
 susceptance, $q_{k,l,1}, q_{k,l,2} \in [-1,1]$
and  $c_{k,l,1}, c_{k,l,2} \in \R$. Note that with a slight of abuse of
notation the vector $\bqq$ consists of all the stochastic random variables
$\{q_{k,l,1}, q_{k,l,2}\}_{(k,l) \in \mcA}$.


 \end{itemize}

For sufficiently small coefficients $c_{k}, c_{k+1}$ with $k = 1, \dots, m$ and
$c_{k,l,1}, c_{k,l,2}$ for all tuples $(k,l) \in \mcA$ the assumptions
of Theorems \ref{regions:theorem3}, \ref{regions:theorem1} and \ref{regions:theorem2}
are satisfied for some initial condition $\bx_0$ and thus we can justify the use of 
the sparse grids.  A more detailed analysis
of the size of these coefficients is left for a future analysis emphasizing the details of
power systems.
 
\begin{figure}[htp]
\begin{center}
\begin{tikzpicture}
\tikzset{cos v source/.style={ circle, draw, scale=1.5, } }

\tikzset{sin v source/.style={
  circle,
  draw,
  append after command={
    \pgfextra{
    \draw
      ($(\tikzlastnode.center)!0.5!(\tikzlastnode.west)$)
       arc[start angle=180,end angle=0,radius=0.425ex] 
      (\tikzlastnode.center)
       arc[start angle=180,end angle=360,radius=0.425ex]
      ($(\tikzlastnode.center)!0.5!(\tikzlastnode.east)$) 
    ;
    }
  },
  scale=1.5,
 }
}
  \draw
(0,0) node [sin v source] (v1) {} 
(v1.north)--++(0,0.5) coordinate(v1-up) 
(v1-up)--++(-0.5,0) 
(v1-up)--++(2,0) node[right]{$1$} 
coordinate[pos=0.5](v1-r) 
coordinate[pos=0.75](v1-rr) 
(v1-rr)--++(0,-0.25)--++(5,0)--++(0,0.25) 
coordinate(v2-ll) 
(v2-ll)--++(-0.5,0) node[left]{$3$} 
(v2-ll)--++(2,0) 
coordinate[pos=0.25](v2-l) 
coordinate[pos=0.75](v2-up) 
(v1-r)--++(0,-0.25)--++(2.625,-2)--++(0,-0.25) 
coordinate(v3-l) 
(v3-l)--++(-0.25,0) 
(v3-l)--++(1,0) 
coordinate[pos=0.375](v3-up) 
coordinate[pos=0.75](v3-r) 
(v3-up) --++(0,-0.25) node[below,cos v source]{} 
(v3-r) --++(0,0.25)--++(2.625,2)--(v2-l) 
;
\draw[-stealth](v2-up)--++(0,-0.5cm); 
\node at (0.5,-.5) [] {$S_{G}$};
\node at (6,-1.5) [] {$V_2 = 1.05$};
\node at (-1,0.98) [] {$1 \phase{0^{\circ}}$};
\node at (8.5,-.5) [] {$\begin{matrix*}[l]
      S_{L_2} &= P_{L} + iQ_{L} \\
      &=(c_2q_2 + 2.8653) \\
      &+ i(c_3
      q_3 + 1.2244)
      \end{matrix*}
      $};
\node at (6,-2.5) [] {$P_{E}=c_1q_1 + 0.6661$};
\node at (4,-1.5) [] {$2$};
\end{tikzpicture}
\end{center}
\caption{2 generators, 3 buses, 1 load simple power system example. This
figure is modified from \cite{Fiandrino2013}. Bus 1 is the slack
bus. Bus 2 contains a stochastic generator. Bus 3 contains the random
load. Note that volatages and power flows are in {\it p.u.}}
\label{example1:fig1}
\end{figure}
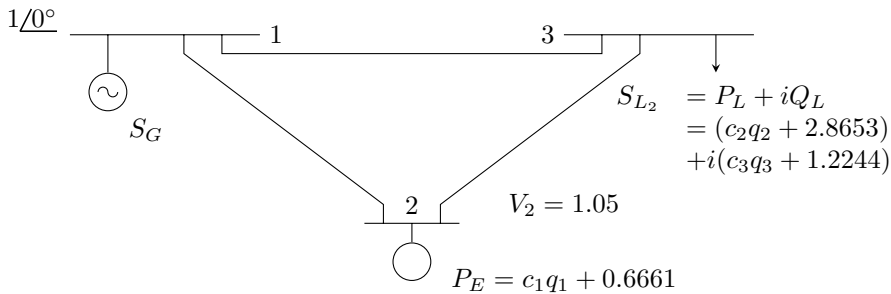


We test the sparse grid approximation on the New England 39 Bus, 10
Generator, power system model provided from the Matpower 6.0 steady
state simulator \cite{Zimmerman2011, Murillo2013}. In this model buses
1 - 29 are PQ buses, buses 30, 32-39 are generators and bus 31 is the
reference (slack).

Two tests are performed. We randomly perturb either the loads or the
admittances of the transmission lines. The mean and variance of the
voltage $V_{22}$ at bus 22 are computed.  The mean
$\mathbb{E}[V_{22}]$ and variance $\var[V_{22}]$ are computed with the
Clenshaw-Curtis isotropic {\it Sparse Grid Matlab Kit}
\cite{Back2011,Tamellini2015} for $N = 2,4,12$ dimensions and up to the $w
= 7$ level. This last level, $w = 7$ is taken as the ``true'' solution.
The errors are computed up to level $w = 4$ with respect to this solution. Two tests are performed:

\begin{itemize}
\item {\bf Random loads:} The loads are considered stochastic and are
  perturbed by up to $\pm$ 50\% of their nominal value. For each $k =
  1,\dots, N$, the so-called $k^{th}$ PQ bus is stochastically perturbed as
  \[
  P_{k} + iQ_{k} = P^{0}_k(1 + \frac{q_k}{2}) + Q^{0}_k(1 + \frac{q_k}{2}),
  \]
  where $P^{0}_k$ and $Q^{0}_k$ are the nominal power loads, $q_k \in
  [-1,1]$, $\rho(q_k)$ has a uniform distribution and the random
  variables $q_1,\dots,q_N$ are independent. Note that although the
  load random perturbations are independent, the power flows will be
  dependent on all the random variables $q_1,\dots,q_N$.

  In Figure \ref{numericalresults:fig1} (a) \& (b) the mean and
  variance convergence error for the stochastic voltage $V_{22}$ of
  bus 22 are shown. A surrogate model based on the sparse grid
  operator is formed as $\mcS^{m,g}_{w} [V_{22}]$ with
  Clenshaw-Curtis abscissas. Each of the circles corresponds to a
  sparse grid $\mcS^{m,g}_{w}$ starting with level $w = 1$ up to
  level $w = 4$.  The y-axis corresponds to the error of the mean or
  variance. The x-axis is the number of sparse grid knots 
  needed to form the grid $\mcS^{m,g}_{w}$. The dimension of
  the sparse grid is given by $N = 2,4,12$.

  From Figure \ref{numericalresults:fig1} (a) \& (b) we observe
  that the error decreases faster than polynomially with respect to the
  number of knots $\eta$. As we increase the number $w$ of levels, sub-exponential convergence is achieved. This is much faster than
  the $\eta^{-\frac{1}{2}}$ convergence rate of the Monte Carlo
  method. However, as the number of dimensions $N$ increases the
  convergence rate of the sparse grid decreases, as predicted by
  Theorem \ref{erroranalysis:theorem1}. Moreover, if the level $w$ is
  not large enough then the error bound gives algebraic convergence.

\item {\bf Random transmission line admittances:} The admittances of the
  network are assumed to be random with
 \[
 Y_{k,l} =  G_{k,l} + iB_{k,l} = G^{0}_{k,l}(1 + \frac{q_{k,l,1}}{2}) + B^{0}_{k,l}
 (1 + \frac{q_{k,l,2}}{2}),
 \]  
 where $G^{0}_{k,l}$ and $B^{0}_{k,l}$ are the nominal conductance and
 susceptance.  The coefficients $q_{k,l,1}, q_{k,l,2} \in [-1,1]$ have a uniform
 distribution and are all independent. 
 Figure
 \ref{numericalresults:fig1} (c) \& (d) indicate sub-exponential
 convergence of the mean and variance of the voltage $V_{22}$ at bus
 22 for a sufficiently large number of knots. However, as the number
 of stochastic dimensions $N$ increases to 12, the
 convergence rate decreases and almost approaches polynomial
 convergence. From Theorem \ref{errorestimates:theorem} the sufficient
 condition $w > N / \log{2}$ leads to subexponential convergence. For
 $N = 12$ we have that $w$ has to be larger than $18$ to guarantee
 sub-exponential convergence. In Figures (c) \& (d) the largest level
 for $w$ is 4.
  \end{itemize}

\begin{figure}[htb]
  \centering
  \begin{tikzpicture}[thick,scale=0.9, every node/.style={scale=0.9}]
\node[inner sep=0pt] (russell) at (0,0)
    {\includegraphics[width=1.02\textwidth]{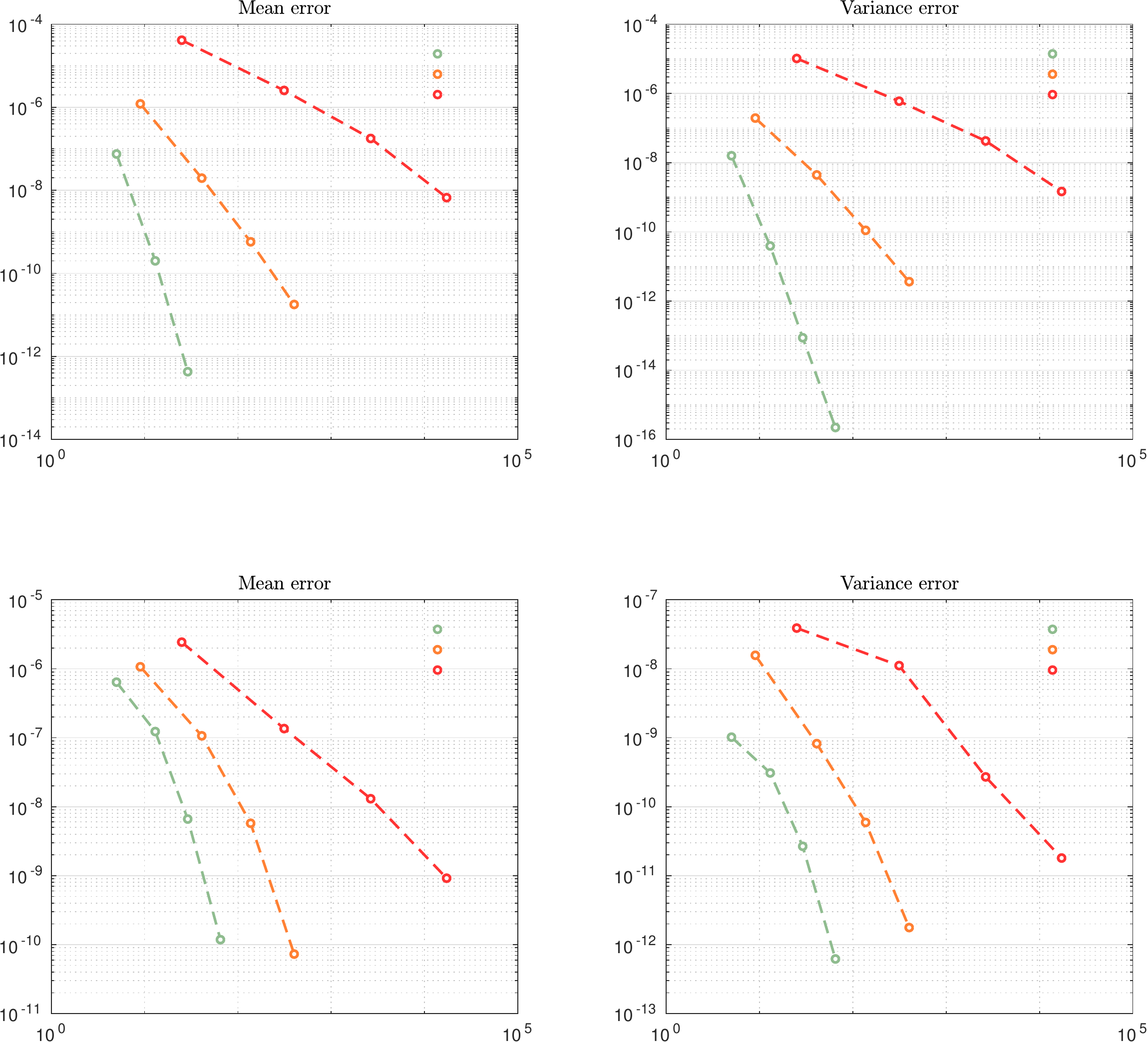}};
    \node [] at (-4.0,0.8) { $\eta$ (knots)};
    \node [] at (4.6,0.8) { $\eta$ (knots)};

    \node [] at (-4.0,-7.3) { $\eta$ (knots)};
    \node [] at (4.6,-7.3) { $\eta$ (knots)};

    \node [] at (-4.0,0.3) {(a)};
    \node [] at ( 4.6,0.3) {(b)};
    \node [] at (-4.0,-8.0) {(c)};
    \node [] at ( 4.6,-8.0) {(d)};

    \node [rotate=90] at ( -8.4,4.5) {$|\bbE[V_{22}] -
      \bbE[\mcS^{m,g}_{w}[V_{22}]]|$};

    \node [rotate=90] at ( 0,4.5) {$|\var[V_{22}] -
      \var[\mcS^{m,g}_{w}[V_{22}]]|$};

        \node [rotate=90] at ( -8.4,-4.5) {$|\bbE[V_{22}] -
      \bbE[\mcS^{m,g}_{w}[V_{22}]]|$};

    \node [rotate=90] at ( 0,-4.5) {$|\var[V_{22}] -
      \var[\mcS^{m,g}_{w}[V_{22}]]|$};

    \footnotesize
    
    \node [] at (-1.35,6.7) {$N  = 2$};
    \node [] at (-1.35,6.35) {$N  = 4$};
    \node [] at (-1.35,6.0) {$N = 12$};

    \node [] at (7.35,6.7) {$N = 2$};
    \node [] at (7.35,6.35) {$N = 4$};
    \node [] at (7.35,6.0) {$N = 12$};

    \node [] at (-1.35,-1.5) {$N =  2$};
    \node [] at (-1.35,-1.8) {$N =  4$};
    \node [] at (-1.35,-2.1) {$N = 12$};

    \node [] at (7.35,-1.5) {$N =  2$};
    \node [] at (7.35,-1.8) {$N =  4$};
    \node [] at (7.35,-2.1) {$N = 12$};

    \normalsize
\end{tikzpicture}
\caption{Sparse grid convergence rates. (a) \& (b) Mean and variance
  error of the voltage $V_{22}$ of bus 22 given a stochastic load
  perturbation with dimension $N$ and the number of knots of the
  sparse grid. (c) \& (d) Mean and variance of error of the voltage
  $V_{22}$ of bus 22 given a random admittance with dimension $N$. Notice that for all 4 cases
  the convergence rates are faster than polynomial, indicating a
  sub-exponential convergence rate.}
\label{numericalresults:fig1}
\end{figure}
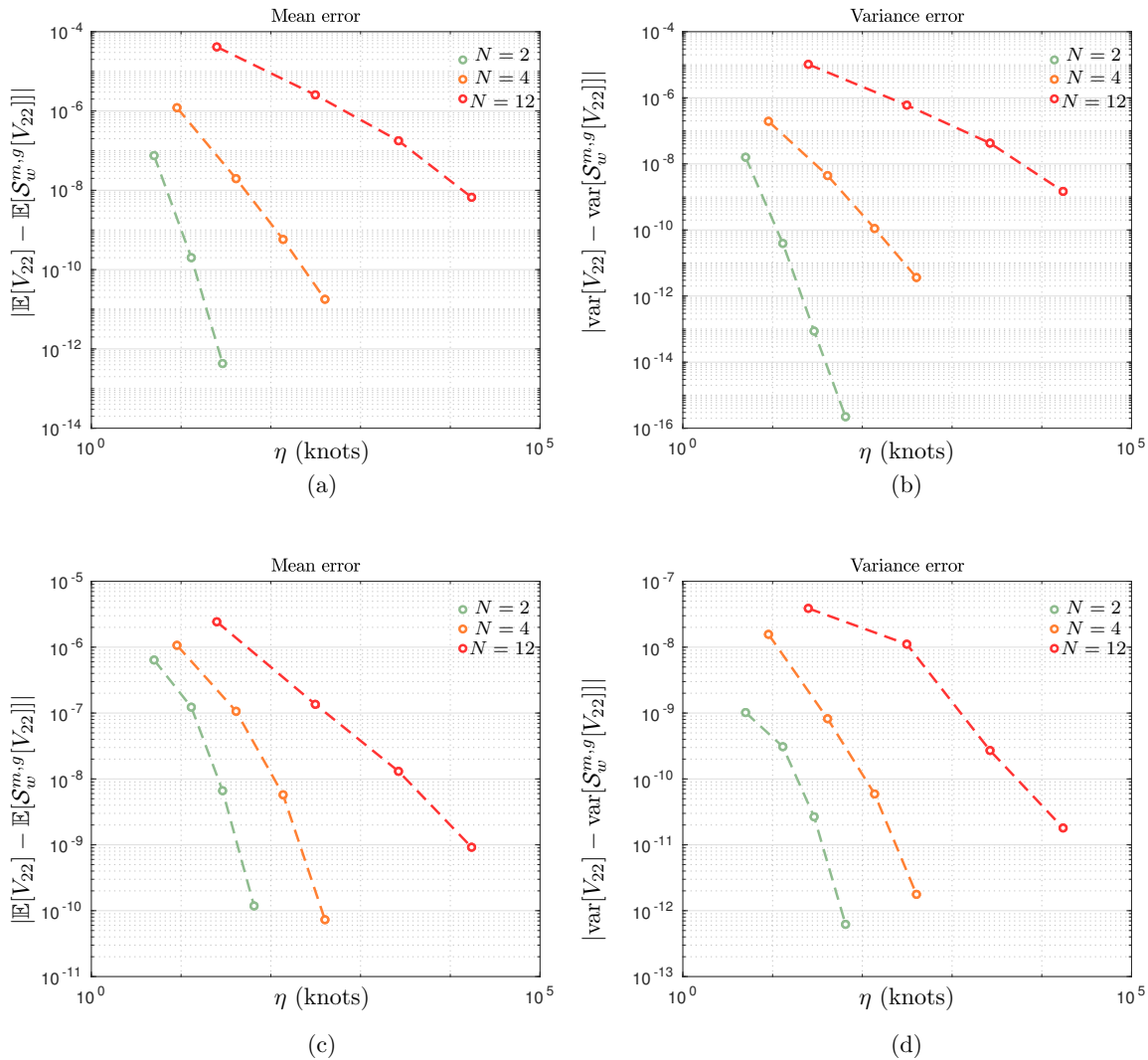

\section{Conclusions}

In this paper we have introduced ideas from UQ and numerical analysis
for the solution of stochastic PDEs using the Newton iteration. In
particular we have developed a regularity analysis of the solution
with respect to the random perturbations. Under sufficient conditions
based on the Newton-Kantorovich Theorem there exists analytic
extensions of the solution of the Newton iteration.  These
indicate that the application of sparse grids for the computation of
the stochastic moments leads to sub-exponential or algebraic
convergence. For a moderate number of dimensions the convergence rates
are much faster than traditional Monte Carlo approaches
($\eta^{-\frac{1}{2}}$). In addition, numerical experiments applied to
the power flow problem confirm these subexponential and algebraic convergence rates.

A weakness in the application of the Newton-Kantorovich Theorem is
that is constricts the size of the region of analyticity $\Psi$, thus
leading to a conservative convergence rate of the sparse grid. This
motivates the application of less restrictive methods such as damped
Newton iterates \cite{Bank1982}. In addition, if we incorporate the 
assumption that all the Newton
iterations converge for each of the knots of the space grid, then by
developing an a posteriori method convergence rates can be further
improved.

Future work includes the important application of this method to the
security constrained problem \cite{Roald2013} from the probabilistic
perspective. In other words, given stochastic perturbations of the
loads and sources what are the optimal power injections into the grid
such that the probability of failure is below a tolerance
level. Current approaches rely on simplifications of the stochastic
perturbations to deal with the high dimensions. However, this can lead
to suboptimal results. The high dimensional stochastic quadrature
approach developed in this paper will allow more optimal results.



    





\bibliographystyle{plain}
\bibliography{multilevel,citations,securityconstrained}

\end{document}